\documentclass{amsart}
\usepackage[utf8]{inputenc}
\usepackage{amssymb} 
\usepackage{todonotes}
\usepackage{xfrac} 
\usepackage{hyperref}
\usepackage{xcolor}
\usepackage{enumitem}
\usepackage{tabto}
\usepackage{tikz-cd}

\hypersetup{
    colorlinks=true,    
    linkcolor=blue,          
    citecolor=blue,      
    filecolor=blue,      
    urlcolor=blue           
}

\title[On the boundedness of singularities via normalized volume]{On the boundedness of singularities via normalized volume}

\author[Y.~Liu]{Yuchen Liu}
\address{Department of Mathematics, Northwestern University, Evanston, IL 60208, USA
}
\email{yuchenl@northwestern.edu}

\author[J.~Moraga]{Joaqu\'in Moraga}
\address{UCLA Mathematics Department, Box 951555, Los Angeles, CA 90095-1555, USA
}
\email{jmoraga@math.ucla.edu}

\author[Hendrik S\"u{\ss}]{Hendrik S\"u{\ss}}
\address{Friedrich-Schiller-Universität Jena,
Institut für Mathematik,
Ernst-Abbe-Platz 2
07743 Jena,
Germany.}
\email{hendrik.suess@uni-jena.de}

\newtheorem{theorem}{Theorem}[section] 

\newtheorem{thm}[theorem]{Theorem}
\newtheorem{lemma}[theorem]{Lemma}
\newtheorem{lem}[theorem]{Lemma}
\newtheorem{proposition}[theorem]{Proposition}

\newtheorem{corollary}[theorem]{Corollary}

  \newtheorem{introthm}{Theorem}
  \newtheorem{introconj}{Conjecture}

\newtheoremstyle{boldremark}
    {}
    {}
    {}          
    {}          
    {\bfseries} 
    {.}         
    {.5em}      
    {}          

\theoremstyle{boldremark}
\newtheorem{remark}[theorem]{Remark}
\newtheorem{rem}[theorem]{Remark}
\newtheorem{example}[theorem]{Example}
\newtheorem{notation}[theorem]{Notation}

\newcommand{\X}{\mathcal{X}}
\newcommand{\Y}{\mathcal{Y}}
\newcommand{\dC}{\Sigma_\X}
\newcommand{\TX}{T\X}
\newcommand{\m}{\mathfrak{m}}

\newcommand{\pp}{\mathbb{P}}

\newcommand{\qq}{\mathbb{Q}}
\newcommand{\zz}{\mathbb{Z}}
\newcommand{\nn}{\mathbb{N}}
\newcommand{\rr}{\mathbb{R}}

\newcommand{\kk}{\mathbb{K}}

\newcommand{\D}{\mathcal{D}}

\newcommand{\TT}{\mathbb{T}}

\newcommand{\A}{\mathcal{A}}

\newcommand{\aX}{a_X}
\newcommand{\aXD}{a_{(X,\Delta)}}

\newcommand{\dfan}{\mathcal{S}}
\newcommand{\I}{\mathcal{I}}

\theoremstyle{definition}
\newtheorem{construction}[theorem]{Construction}
\newtheorem{definition}[theorem]{Definition}

\newcommand{\CC}{\mathbb{C}}
\newcommand{\CS}{\mathcal{S}}
\newcommand{\CO}{\mathcal{O}}

\newcommand{\RR}{\mathbb{R}}
\newcommand{\QQ}{\mathbb{Q}}
\newcommand{\ZZ}{\mathbb{Z}}
\newcommand{\NN}{\mathbb{N}}
\newcommand{\PP}{\mathbb{P}}
\newcommand{\facets}{\mathbf{F}}
\newcommand{\Hilb}{\mathcal{H}^{X_0,\TT}}
\newcommand{\Hilbp}{\mathcal{H}^{(X_0,B_0),\TT}}
\newcommand{\Hilbs}{\mathcal{H}}

\DeclareMathOperator{\Loc}{Loc}
\DeclareMathOperator{\lspan}{span}
\DeclareMathOperator{\supp}{supp}
\DeclareMathOperator{\ord}{ord}
\DeclareMathOperator{\lcm}{lcm}
\DeclareMathOperator{\rint}{int}
\DeclareMathOperator{\tail}{tail}
\DeclareMathOperator{\vol}{vol}
\newcommand{\pos}{\RR_{\geq 0}\cdot}
\DeclareMathOperator{\conv}{conv}

\DeclareMathOperator{\Hom}{Hom}
\DeclareMathOperator{\Spec}{Spec}

\DeclareMathOperator{\Val}{Val}
\DeclareMathOperator{\nvol}{\widehat{vol}}
\DeclareMathOperator{\dist}{dist}
\DeclareMathOperator{\mld}{mld}
\DeclareMathOperator{\mldK}{mldK}

\newcommand{\bR}{\mathbb{R}}
\newcommand{\bA}{\mathbb{A}}
\newcommand{\bQ}{\mathbb{Q}}

\newcommand{\bG}{\mathbb{G}}

\newcommand{\bT}{\mathbb{T}}

\newcommand{\hvol}{\widehat{\mathrm{vol}}}
\DeclareMathOperator{\lct}{lct}

\newlength\mylen
\newlist{mycases}{enumerate}{1}
\setlist[mycases,1]{label=\textit{Case~\arabic*.}, labelwidth=*,labelindent=0pt, wide,itemsep=1em}

\newcommand\litem[1]{\item{#1: \hspace{1ex}}}


\begin{document}
\begin{abstract}
In this article we study conjectures regarding normalized volume
and boundedness of singularities.
We focus on singularities with a torus action of complexity $1$,
threefold singularities, and
hypersurface singularities.
Given a real value $v>0$, 
we prove that the class of K-semistable threefold singularities
with normalized volume at least $v$
forms a bounded family. 
Analogous statements are proved
in the case of $n$-dimensional complexity-$1$ 
and $n$-dimensional hypersurface singularities 
for arbitary $n$.
In the general case of klt singularities, i.e. without the assumption on K-semistability,
we show that, up to special degenerations, 
the normalized volume bounds singularities with a complexity\nobreakdash-$1$ torus action.
We exhibit a $3$-dimensional example
which shows that this last statement is optimal.
\end{abstract}
\maketitle

\setcounter{tocdepth}{1} 
\tableofcontents

\section{Introduction}

The study of 
singularities 
has played a fundamental role 
in the development of birational geometry~\cite{Xu17}
and K-stability~\cite{Xu20}.
The singularities that arise
in both theories are known as 
log terminal singularities (or klt). 
The normalized volume is an 
important invariant for klt singularities
which plays a fundamental role 
in the context of K-stability~\cite{LLX20}.
K-semistable singularities are a special class
of Fano cone singularities among klt singularities
(see Definition~\ref{def:K-ss}).
It is conjectured that
the normalized volume, denoted by $\nvol$, 
bounds K-semistable singularities (see e.g. \cite{han2020acc, Zhu21, XZ22}). 

\begin{introconj}\label{conj:boundedness-sing}
Let $d$ be a positive integer. 
Let $\epsilon$ and $\eta$ be a positive real numbers. 
The class of $d$-dimensional K-semistable singularities $(X,\Delta;x)$ for which  
$\nvol(X,\Delta;x)>\epsilon$
and ${\rm coeff}(\Delta)>\eta$ forms a log bounded family.
\end{introconj}

The previous conjecture is known to hold
for exceptional singularities (see \cite{HLS19}) and toric singularities (see, e.g.,~\cite{MS21}).
In this article, we prove a slightly weaker version of this conjecture in dimension three.

\begin{introthm}\label{introthm:3-dim}
Let $\epsilon$  be a positive real number and $I \subset [0,1) \cap \QQ$ a finite set.
The class of $3$-dimensional K-semistable singularities $(X,\Delta;x)$ with $\Delta$ having coefficients in $I$ and $\nvol(X,\Delta;x)>\epsilon$ forms a log bounded family.
\end{introthm}

Together with the stable degeneration theorem  \cite{XZ22} (see Theorem \ref{thm:sdt}), Theorem \ref{introthm:3-dim} implies the discreteness of normalized volumes in dimension $3$, answering affirmatively \cite[Conjecture 1.1(1)]{han2020acc} in dimension $3$ for rational coefficents.

\begin{introthm}\label{introthm:vol-discrete}
    Let $I\subset [0,1)\cap \QQ$ be a finite set. Then the following set
    \[
    \left\{\hvol(X,\Delta;x)\left| \begin{array}{l} (X,\Delta;x) \textrm{ is a $3$-dimensional klt singularity,} \\ \textrm{and $\Delta$ has coefficients in $I$.}\end{array}\right.\right\}
    \]
    has $0$ as its only accumulation point.
\end{introthm}

Note, that the above results have been obtained recently and independently in \cite{zhuan23}. 
In order to prove Theorem~\ref{introthm:3-dim},
we use the language of optimal destabilizations~\cite{BLZ19, LXZ21}.
These are a special class of degenerations for K-unstable Fano type varieties.
In our setting, the special fibers of these degenerations are K-semistable $3$-dimensional singularities with normalized volume away from zero and complexity at most one, i.e., 
they admit the action of either a $2$-dimensional torus or a $3$-dimensional torus.
Thus, we reduced the proof of Theorem~\ref{introthm:3-dim} to two pieces: proving Conjecture~\ref{conj:boundedness-sing} for complexity one $3$-dimensional singularities and 
showing a boundedness statement for K-semistable deformations of K-semistable singularities.
In Proposition~\ref{prop:log-bounded-deformation+k-semistable}, we prove the latter statement. 
This proof uses techniques from torus invariant families, special degenerations, and multigraded Hilbert schemes (see Section~\ref{sec:multi} and Section~\ref{sec:torus-inv-fam}).

In the context of complexity one singularities, we show that the conjecture holds in arbitrary dimension.

\begin{introthm}\label{thm:n-dim-K-ss-comp-1}
Let $\epsilon$ be a positive real number, $d \in \mathbb N$ and $I \subset [0,1) \cap \QQ$ a finite set.
The class of $d$-dimensional K-semistable singularities $(X,\Delta;x)$ of complexity one 
for which the normalized volume $\nvol(X,\Delta;x)>\epsilon$ and ${\rm coeff}(\Delta) \in I$ forms a log bounded family.
\end{introthm} 

When working with general klt singularities, 
it is also expected that the normalized volume 
can be used to bound these singularities~\cite{han2020acc} in a weaker sense.
It is conjectured that klt singularities whose
normalized volume is bounded away from zero, 
are bounded up to special degeneration.
This conjecture is closely related
to the boundedness
up to deformation
of exceptional singularities~\cite{Mor18c,HLM20}.
In~\cite{han2020acc}, the authors prove the conjecture for terminal $3$-fold singularities.
In~\cite{MS21}, the authors prove the conjecture for toric singularities using results from convex geometry such as the Blaschke-Santal\'o inequality.
However, in such a case, we only get finitely many singularities. 
In~\cite{Zhu21}, the author draws a connection between the boundedness using normalized volume
and minimal log discrepancies of Koll\'ar components.
Koll\'ar components are a special kind of exceptional divisors over klt singularities~\cite{Mor18b}.
In~\cite{Zhu21}, Zhuang proves the boundedness up to special degeneration for $3$-fold singularities normalized volume bounded away from zero. This observation is crucially important for our work. In fact, as
an important step in our proof of Theorem~\ref{thm:n-dim-K-ss-comp-1}, we show this weaker form of log boundedness for the case of complexity one $\mathbb{T}$-singularities.

\begin{introthm}\label{introthm:norm-vol-bound}
Let $d$ be a positive integer, $\epsilon$ a positive real numbers and $I \subset [0,1) \cap \QQ$ a finite set.
The class of $d$-dimensional log terminal
complexity one $\mathbb{T}$-singularities $(X,\Delta;x)$
for which $\nvol(X,\Delta;x)>\epsilon$
and ${\rm coeff}(\Delta) \in I$
is log bounded up to special degeneration.
\end{introthm}

Examples show that we only get boundedness \emph{up to special degenerations}, but not boundedness in the usual sense (see Example~\ref{ex:not-bounded}). This shows that without further assumption the outcome of Theorem~\ref{introthm:norm-vol-bound} is optimal. We recall that bounded up to special degeneration
means that all the singularities as in the statement are deformations of a bounded set of algebraic singularities.

In order to prove Theorem~\ref{introthm:norm-vol-bound} by applying the results of \cite{Zhu21} we need to bound the {\em minimal log discrepancies}  of $\mathbb{T}$-singularities.
This invariant known as mld
is tightly related to the termination of flips~\cite{HM20}.
Shokurov conjectured that this invariant 
satisfies the ascending chain condition in a fixed dimension~\cite{Sho04}.
On the other hand, Ambro conjectured that this invariant is lower semicontinuous~\cite{Amb06}.
Both conjectures, independently, imply that in a fixed dimension there is an upper bound for the minimal log discrepancy.
We prove the existence of an upper bound for the minimal log discrepancy of
log terminal complexity one $\mathbb{T}$-singularities among Koll\'ar components.

\begin{introthm}\label{introthm:upper-bound-mld}
Let $d$ be a positive integer
and $\eta$ be a positive real number.
There exists a positive integer $A(d,\eta)$, 
only depending on $d$ and $\eta$, satisfying the following.
Let $(X,\Delta;x)$ be a $d$-dimensional log terminal 
complexity one $\mathbb{T}$-singularity for which
${\rm coeff}(\Delta)>\eta$.
Then, we have that 
$a_{(X,\Delta)}(E) \leq A(d,\eta)$ for some Koll\'ar component $E$ of $(X,\Delta;x)$.
\end{introthm}

In the proof of the previous theorem, 
we will use the language of polyhedral divisors
and toric degenerations to control
the minimal log discrepancy.

Finally, in Section~\ref{sec:hyper}, we complement our previous results
by proving Conjecture~\ref{conj:boundedness-sing} in the case of hypersurface singularities.

\begin{introthm}\label{introthm:hypersurfaces}
Let $d$ be a positive integer and $v$ be a positive real number. 
The class of $d$-dimensional K-semistable hypersurface singularities $(X;x)$ with $\nvol(X;x)>v$ forms a bounded family.
\end{introthm}

In order to prove the previous theorem, 
we will use the Lichnerowicz inequality
as well as some degeneration arguments.

\subsection*{Acknowledgements}
We would like to thank Ziquan Zhuang for informing us that he obtained Theorems~\ref{introthm:3-dim} and \ref{introthm:vol-discrete} independently by using different techniques and making his manuscript available to us, see \cite{zhuan23}.
The authors would also like to thank Nathan Owen Ilten for many useful comments, and Chenyang Xu for helpful discussions. The first author was partially supported by NSF DMS-2148266 (formerly DMS-2001317) and an Alfred P. Sloan
research fellowship.  The third author was supported by the EPSRC grants EP/V013270/1, EP/V055445/1 and by funding from the Carl Zeiss Foundation.

\section{Preliminaries}
\label{sec:prel}

In this subsection, we introduce some preliminaries regarding 
the singularities of the minimal model program, normalized volume, Fano cone singularities, varieties with torus actions, and special degenerations of singularities.

\subsection{Singularities of the MMP}
In this subsection, we recall the singularities of the minimal model program
and purely log terminal blow-ups.

\begin{definition}
A \emph{pair} is a couple
$(X,\Delta)$ such that
$X$ is a normal quasi-projective variety
and $\Delta$ is an effective divisor for which
$K_X+\Delta$ is $\rr$-Cartier.
Let $x\in X$ be a closed point.
We write $(X,\Delta;x)$ for the corresponding singularity.
We say that $(X,\Delta;x)$ is a singularity of the pair.

Let $(X,\Delta;x)$ be a singularity of a pair.
Let $\pi\colon Y\rightarrow X$ be a projective birational morphism and 
$E\subset Y$ be a prime divisor. 
The {\em log discrepancy} of $(X,\Delta)$ at $E$ is the rational number:
\[
a_{(X,\Delta)}(E):=1+{\rm coeff}_E(K_Y-\pi^*(K_X+\Delta)).
\]
We say that $(X,\Delta;x)$ is {\em Kawamata log terminal} 
(or {\em log terminal}) 
if all the log discrepancies of $(X,\Delta)$ are positive on a neighborhood of $x\in X$.
We may write {\em klt} or {\em log terminal}
for short.
We write $a_E(X)$ if $\Delta=0$.
The {\em minimal log discrepancy} of $(X,\Delta)$ at $x$ is defined to be
\[
\mld_x(X,\Delta):=\min\{a_{(X,\Delta)}(E) \mid c_X(E)=x\}. 
\]
If $(X,\Delta;x)$ is klt, then the minimal log discrepancy of $(X,\Delta)$ at $x$ is positive.

In our paper the $\TT$-equivariant minimal log discrepancies \emph{among Koll\'ar components} will play a major role. It is defined as follows.
\[
\mldK_x^\TT(X,\Delta)=\min\{a_{(X,\Delta)}(E) \mid c_X(E)=x, \text{ $E$ is a $\TT$-invariant Koll\'ar component}\}. 
\]
\end{definition}

\begin{definition}
Let $(X,\Delta)$ be a log pair.
We say that $(X,\Delta)$ has {\em purely log terminal singularities}
if the following conditions hold:
\begin{itemize}
    \item all the exceptional log discrepancies of $(X,\Delta)$ are positive, and 
    \item the coefficients of $\Delta$ are in the set $[0,1]$.
\end{itemize}
To abbreviate, we write {\em plt} instead of 
purely log terminal.
\end{definition}

  \begin{definition}
    Let $(X;x)$ be  a klt singularity. We say that a proper birational morphism $f \colon  \widetilde{X} \to  X$ induces
    a {\em weak component} $E$ if the following conditions hold:
    \begin{enumerate}
        \item the morphism $f$ is an isomorphism
        over $X\setminus \{x\}$, 
        \item the fibre $f^{-1}(x)$ is an irreducible divisor $E$, and 
        \item we have $-E$ is $f$-ample. 
    \end{enumerate}
    We say that $f\colon \widetilde{X}\rightarrow X$ induces a {\em Koll\'ar component} $E$ if furthermore
    the pair $(\widetilde{X},E)$ is plt.
  \end{definition}

\subsection{Normalized volume}

In this subsection, we recall the definition of normalized volume.
First, we start with the definition of the volume of a valuation over a singularity (see, e.g.,~\cite{ELS03}).

\begin{definition}
Let $(X;x)$ be a normal algebraic singularity of dimension $d$.
Let $v\in \Val(X;x)$ be a valuation centered at $x\in X$.
We write $a_m(v):=\{ f\in \mathcal{O}_{X,x} \mid 
v(f)\geq m\}$.
The {\em volume}
of $v$ is defined to be 
\[
\vol_X(v):=\limsup_{m\rightarrow\infty} 
\frac{ 
\ell(\mathcal{O}_{X,x}/a_m(v))
}{m^d/d!}
\]
where $\ell$ is the length of the module.
\end{definition}

The following is a special case of a definition
due to Li~\cite{Li18}.

\begin{definition}\label{def:norm-volume}
Let $(X,\Delta;x)$ be a $d$-dimensional
log terminal singularity of a pair.
The {\em normalized volume function}
is the function
\[
\nvol_{(X,\Delta)}
\colon \Val(X;x)\rightarrow \rr_{>0}\cup \{\infty\} 
\]
defined by the formula
\[
\nvol(v):=
a_{(X,\Delta)}(v)^d \vol_X(v) 
\] 
whenever $a_{(X,\Delta)}(v)\in \rr_{>0}$
and as $\infty$ otherwise.
The {\em normalized volume} (also called the \emph{local volume} in some literature) of 
a  log terminal singularity $(X,\Delta;x)$ of a pair
is defined to be 
\[
\widehat{\rm vol}(X,\Delta;x):=
\inf\left\{\nvol(v)
\mid v\in \Val(X;x)\right\}.
\] 
\end{definition}

\begin{example}
  \label{exp:hypersurfaces1}
   Let $X_d\subset \mathbb A^{n+1}$  be a quasi-homogeneous hypersurface of degree $d$ with respect to the weights
  $w_0, \ldots, w_n > 0$. Then the order with respect to the grading defines a valuation $v$ on $(X_d;0)$ via \(v\left(\sum_u f_u\right) = \min \{u \mid f_u \neq 0\}.\)
  In this situation we have
  \[\nvol(v)= \frac{d\left(\sum_i w_i-d\right)^n}{w_0 \cdots w_n}.\]
  More generally for complete intersections in $\mathbb A^{n+c}$ one has 
  \[\nvol(v)= \frac{\left(\prod_j d_j\right) \cdot \left(\sum_i w_i-\sum_j d_j\right)^n}{\prod_i w_i}.\]
  This follows from the more general considerations in \cite{Li18}[§2.2].
\end{example}

\subsection{Normalized volume and K-semistability of Fano cone singularities}
In this subsection, we recall the connection between normalized volume and the notion of K-semistability.

\begin{definition}
  Let $X$ be an affine variety with an effective action of an algebraic torus $\mathbb{T}$. If $M$ is the character lattice of the torus and $\omega \subset M_\RR:=M\otimes \RR$ is the weight cone of that action. If $\omega$ is a full-dimensional pointed cone and the only invariant regular functions are the constant functions, then there is a unique fixed point $x \in X$ and the pair $(X;x)$ is called a \emph{cone singularity}. We will often surpress $x$ in the notation and also call $X$ itself a cone singularity.

  If additionally $(X;x)$ is log terminal, then it is called a \emph{Fano cone singularity} or \emph{Fano cone $\TT$-singularity} if we want to emphasise the acting torus. More generally, if $\Delta$ on $X$ is a $\TT$-invariant $\QQ$-divisor and $(X,\Delta;x)$ is klt, then it is called a  \emph{log Fano cone singularity}.


  The polyhedral cone $\sigma=\omega^\vee \subset N_\RR =M_\RR^*$ is called the Reeb cone of  $(X;x)$. Since, the fixed point is unique in our setting we will sometimes refer to $X$ itself as a cone singularity, as well.

\smallskip
An element $\xi$ from the interior of \[\sigma = \omega^\vee
= \{v \in M^*_\RR \mid \forall_{u \in \omega} \colon \langle v,u  \rangle \geq 0\}
  \subset N_\RR := M_\RR^*\] is called a \emph{polarization} of $(X;x)$.

\smallskip

Such a polarization $\xi$ can be interpreted as a $\mathbb{T}$-invariant valuation $\xi \in \Val(X;x)$ via $\xi(f_u):=\langle u, \xi \rangle$ for a homogeneous function $f_u$ of weight $u \in M$. 
\end{definition}

We now turn to a relative version of cone singularities.
\begin{definition}
Consider a variety $\mathcal{X}$ with a $\TT$-action, together with a surjective morphisms $f \colon \mathcal{X} \to S$ such that $\mathcal{X}$ is affine over some base $S$ and $f_*\CO_{\mathcal{X}}^{\TT} = \CO_S$. Then the $\TT$-action induces a grading on $f_*\CO_{\mathcal{X}}$. If the weights of this grading span a full-dimensional pointed cone, then we call $\mathcal{X}\to S$ a $\TT$-invariant family of cone singularities and the dual cone of the weight cone is called the \emph{common Reeb cone} of that family.
\end{definition}

\begin{definition}
  A \emph{special degeneration} of a singularity $(X,\Delta;x)$ is a flat family $(\mathcal{X},\mathbf{\Delta};s) \to \mathbb{A}^1$ together with a $\CC^*$-action on $\mathcal{X}$ fixing $\mathbf \Delta$ and with $s \colon \mathbb A^1 \to \X$ being a $\CC^*$-equivariant section,  such that
\begin{enumerate}
\item the structure map of the family is  equivariant with respect to the standard $\CC^*$-action on $\mathbb{A}^1$, and
\item the generic fibre $(\X_t,\Delta_t;s(t))$ is isomorphic to $(X,\Delta;x)$ for $t\neq 0$ and the fibre $(\mathcal X_0,\Delta_0;s(0))$ over $0 \in \mathbb{A}^1$ is klt. 
\end{enumerate}

If $(\mathcal X,\mathbf{\Delta})  \cong (X,\Delta) \times \mathbb{A}^1$, then we call $\mathcal X$ a \emph{product degeneration}. 

Given a $\mathbb{T}$-action on $X$ we say that the special degeneration is $\TT$-invariant, if there is also a $\mathbb{T}$-action on $\mathcal{X}$, fixing $\mathbf{\Delta}$ and restricting to the original $\mathbb{T}$-action on the generic fibre and commuting with the $\CC^*$-action.
\end{definition}

\begin{remark}
  In the case of a cone $\TT$-singularity with $\TT$-invariant special degeneration both, the fixed point and the  section are unique. We will therefore usually surpress them in our notation.
\end{remark}

Given a klt singularity $(X,\Delta;x)$ with a (possibly trivial) torus action, every equivariant special degeneration is induced by a torus invariant divisorial valuation $v \in \Val(X;x)$ via the Rees construction for the associated filtration of the coordinate ring. The central fibre of this special degeneration is then given as the spectrum  of the associated graded ring. 
Now,  there exists a generalized version of the Rees construction for so-called quasi-monomial valuations $v$, see \cite[Section 2.1]{zbMATH02020243}. When the associated graded ring is finitely generated, then its spectrum $X_0$ is the central fibre of a degeneration and comes naturally with a polarization $\xi_v$ instead of a $\CC^*$-action. We say that $v$ induces a degeneration with polarized central fibre $(X_0, \Delta_0;\xi_v)$.

\begin{proposition}[{\cite[Lemma~2.58]{li2017stability}}]
  \label{prop:nvol-special-fibre}
  Consider a degeneration of a klt singularity $(X,\Delta;x)$ with klt polarized special fibre $(\mathcal X_0,\Delta_0; \xi_v)$  induced by a quasi-monomial valuation $v$ with finitely generated associated graded ring. Then $\nvol_{(X,\Delta)}(v)=\nvol_{(X_0,\Delta_0)}(\xi_v)$ holds.
\end{proposition}

K-(semi-)stability for Fano cone singularities was first defined in \cite{zbMATH06868031} as a criterion for the existence of a Sasaki-Einstein metric on the link of the singularity. The authors already observed that this notion has an interpretation in terms of volume minimization of the central fibre of a special degeneration. Now, the correspondence between valuations and degenerations lets us directly express K-semistability in terms of the normalized volume of valuations. More precisely, we have the following result due to Li and Xu.

\begin{theorem}[{\cite[Theorem 2.34]{li2017stability}}]
  \label{thm:ksstable-minimizer}
  A polarised Fano cone singularity $(X,\Delta;\xi)$ is K-semistable if and only if $\xi$ is a minimizer of $\nvol_{(X,\Delta)}$. 
\end{theorem}
In our paper this characterization will effectively serve as the definition of K-semistability.

More generally, the stable degeneration theorem, proved as a combination of a series of works \cite{Li18, Li17, LL19, Blu18, LWX21, Xu19, XZ21, XZ22}, shows that the normalized volume minimizer of an arbitrary klt singularity induces a K-semistable log Fano cone degeneration.

\begin{thm}[Stable Degeneration Theorem]\label{thm:sdt}
Let $(X,\Delta;x)$ be a klt singularity. Then there exists a valuation $v\in \Val(X;x)$ that minimizes the normalized volume  function $\hvol_{(X,\Delta)}$. Moreover, $v$ is quasi-monomial, unique up to rescaling, has a finitely generated associated ring, and induces a K-semistable log Fano cone degeneration $(X_0,\Delta_0;\xi_v)$.
\end{thm}

\begin{example}
  \label{exp:hypersurface2}
  We consider a hypersurface given by a polynomial $f \in \CC[x_0, \ldots x_n]$ which is homogeneous of weight $\mathbf u$ with respect to an $M$-grading given by the weights $u_0,\ldots,u_n \in M$. Then the weight cone $\omega$ is generated by $u_0,\ldots,u_n$. Take a polarization $\xi \in \sigma = \omega^\vee$ and set $w_i=\langle u_i, \xi \rangle$ and $d=\langle \mathbf u , \xi \rangle$. Then the formula of Example~\ref{exp:hypersurfaces1} extends by continuity of $\nvol$ on the interior of $\sigma$ and we once again obtain
  \[\nvol(\xi)= \frac{d(\sum_i w_i-d)^n}{w_0 \cdots w_n}.\]
  \end{example}

  In the situation of an isolated singularity there is a well-known obstruction to K-semistability, known as Lichnerowicz obstruction, see \cite{zbMATH05346216}. The corresponding inequality for the hypersurface case can be also obtained (even for non-isolated ones) as an easy consequence of Li's theorem.
  \begin{theorem}[{\cite[3.23]{zbMATH05346216}}]
    \label{thm:lichnerowicz}
    Consider a Fano cone hypersurface singularity $(X;x)$ with polarization $\xi$ corresponding to the weight vector $(w_0,\ldots,w_n)$. If $(X,\xi)$ is K-semistable, then the following inequality holds
    \begin{equation*}
      \sum_{i=0}^n w_i -d \leq nw_0.\label{eq:lich-obstr}      
    \end{equation*}
  \end{theorem}
  \begin{proof}
    We consider the special degeneration, which is induced by the valuation $v_s=s \cdot \ord_{x_0}$ with $s>0$, i.e. a positive multiple of the order of vanishing in $x_0$. If $s$ is integral, this is simply the degeneration induced by acting with $\CC^*$ via  the weights $w_0'=w_0+s$, $w_1'=w_1, \ldots, w_n'=w_n$.  Now, by Proposition~\ref{prop:nvol-special-fibre} and Theorem~\ref{thm:ksstable-minimizer} we have
    \[
      \nvol(v_s)=\frac{d(\sum_i w'_i-d)^n}{w'_0 \cdots w'_n}= \frac{d(\sum_i w_i-d+s)^n}{(w_0+s)w_1 \cdots w_n} \geq \frac{d(\sum_i w_i-d)^n}{w_0 \cdots w_n} = \nvol(v_0)
    \]
    for every $s \geq 0$. In particular, we obtain
    \[0 \leq \left.\frac{\partial \nvol(v_s)}{\partial s}\right|_{s=0} \!\!=
      \frac{\left(w_0\cdots w_n\right) \cdot nd(\sum_iw_i-d)^{n-1} - \left(w_1\cdots w_n\right) \cdot d(\sum_iw_i-d)^n}{(w_0\cdots w_n)^2},
    \]
    which gives
    \[
      nw_0 \geq \sum_iw_i -d,
    \]
    as desired.
  \end{proof}

\begin{definition}\label{def:K-ss}
By a {\em K-semistable singularity} we mean a log Fano cone singularity which is K-semistable with respect to a suitable polarization.
\end{definition}

In particular, every K-semistable 
singularity is a klt cone singularity.

\medskip
Crucial for our work is the following result by Zhuang, which bounds Fano cone singularities in terms of normalized volume up to special degenerations.

\begin{theorem}[{\cite[Cor.~4.4]{Zhu21}}]
  \label{thm:zhuan-main}
  Let $\epsilon, A > 0$, $n \in \mathbb{N}$ and let $I \subset [0, 1] \cap \QQ$ be a finite set. Then the set
  of $n$-dimensional klt $\TT$-singularities $x \in (X, \Delta)$ with $\nvol(x, X, \Delta) \geq \epsilon$, $\mldK^{\TT}_x(X, \Delta) \leq A$ and coefficients of $\Delta$ being from $I$ is log bounded up to $\TT$-invariant special degeneration.

  Moreover, we may assume that the bounding family of log Fano cone singularities arises from a $\QQ$-polarized family of projective varieties.
\end{theorem}

\begin{remark}
  Note, that Zhuang does not mention the  $\TT$-invariance of the degeneration and the fact that the bounding family arises from a family of $\QQ$-polarized projective varieties in \cite[Cor.~4.4]{Zhu21}. However, this follows directly from his proof of \cite[Lem.~2.19]{Zhu21}.
\end{remark}

\subsection{\texorpdfstring{$\mathbb{T}$-varieties}{T-varieties}}\label{subsec:pdiv}

\label{def:tvars}
Here, we want to recall some basic facts about $\TT$-varieties and their combinatorial description from \cite{AH06}.

\begin{definition}
A {\em $\mathbb{T}$-variety} is a normal quasi-projective variety $X$ which admits the effective action of an algebraic torus $\mathbb{T}$. 
The {\em complexity} of a $\mathbb{T}$-variety $X$ is
given by the difference $\dim X -\dim T$.
\end{definition}

Note, that the Fano cone singularities introduced in the previous section are a particular case of $\TT$-varieties.

\medskip
In \cite{AH06} a combinatorial description of affine $\TT$-varieties in terms of so-called \emph{polyhedral divisors} was introduced. These divisors live on a residual quotient variety and have coefficient in the semi-group of $\sigma$-polyhedra. Here, we recall the basics of this description.
\begin{definition}
Let $N$ be a free finitely generated abelian group.
We write $M:={\rm Hom}_\zz(N,\zz)$ for its dual.
We denote by $N_\qq$ and $N_\rr$ (resp. $M_\qq$ and $M_\rr$)  the corresponding $\qq$-vector
and $\rr$-vector spaces.
Let $\sigma\subset N_\qq$ be a polyhedral cone.
A {\em $\sigma$-polyhedron} $P$ is a polyhedron in $N_\qq$ for which the tail cone
\[
\tail(P):= \{ v\in N_\qq \mid 
P+v \subset P
\}
\]
equals $\sigma$. The set of $\sigma$-polyhedra together with Minkowski addition forms a monoid with neutral element being $\sigma$ itself.

Let $Y$ be a semi-projective variety, i.e. the natural morphism $p\colon Y \to \Spec(H^0(\CO_Y))$ is projective.
A {\em polyhedral divisor} on $(Y,N_\qq)$ 
is a finite formal sum
\[
\mathcal{D} = \sum_{Z\subset Y} \mathcal{D}_Z \cdot Z,
\]
running over the prime divisors of $Z \subset Y$ where each $\mathcal{D}_Z$ is a $\sigma$-polyhedron or the empty set. By \emph{finite} we mean that only finitely many of the polyhedral coefficients differ from $\sigma$ (the neutral element of the monoid). We call $\supp(\D):=\bigcup_{\D_Z \neq \sigma}Z$ the \emph{support} of $\D$.
The cone $\sigma$ is called the {\em tail} of $\mathcal{D}$.
We write $\tail \D=\sigma$.
The {\em locus} of the polyhedral divisor
$\mathcal{D}$ is the open subset
$Y\setminus \bigcup_{\mathcal{D}_Z=\emptyset} Z$.
We denote this set by
${\rm Loc}(\D)\subset Y$.

Given a polyhedral divisor $\D$ on $(Y,N_\qq)$, we have an associated map
\begin{align*}
\D & \colon  \sigma^\vee \rightarrow {\rm CaDiv}_\qq({\rm Loc}(\D)), \\
u & \mapsto \D(u):=\sum_{Z} \min_{v\in \D_Z}\langle u,v\rangle Z.
\end{align*}
By abuse of notation, we denote by $\D$ both: the polyhedral divisor
and the associated map.

We say that $\mathcal{D}$ is a 
{\em proper polyhedral divisor}
on $(Y,N_\qq)$ if the two following conditions are satisfied:
\begin{itemize}
    \item for each $u\in \sigma^\vee$, we have that $\mathcal{D}(u)$ is semiample on ${\rm Loc}(\D)$, and 
    \item for each $u\in {\rm relint}(\sigma^\vee)$, we have that 
    $\mathcal{D}(u)$ is big on ${\rm Loc}(\D)$.
\end{itemize}
We may write {\em p-divisor} instead of proper polyhedral divisor.
When the $\qq$-vector space $N_\qq$ is clear from the context, we may say that $\D$ is a p-divisor on $Y$.
\end{definition}

To each proper polyhedral divisor on $Y$, we can associate an affine $\mathbb{T}$-variety of complexity $\dim Y$.

\begin{construction}
Let $\mathcal{D}$ be a proper polyhedral divisor on $(Y,N_\qq)$. 
We consider the divisorial sheaf on $Y$ given by
\begin{equation}
\mathcal{A}(\D):= \bigoplus_{u\in \sigma^\vee \cap M} \mathcal{O}_{\Loc(D)}(\D(u)).\label{eq:AH-sheaf}
\end{equation}
The affine variety $X(\D):={\rm Spec}(H^0(\mathcal{A}(\D)))$ is called the \emph{$\mathbb{T}$-variety} associated to $\D$.
Let $\widetilde{X}(\D):= {\rm Spec}_{Y}\mathcal{A}(\D)$
be the associated relative spectrum  with structure map $\pi \colon \widetilde{X}(\D) \to Y$.
\end{construction}

Note that by our construction the weight cone of $H^0(\mathcal{A}(\D))$ is precisely the dual of $\sigma=\tail \D$. Hence, if the resulting variety $X(\D)$ is a cone singularity, as defined in the previous section, then the tail cone $\sigma$ of the polyhedral divisor equals the Reeb cone of the cone singularity or more precisely minus the Reeb cone if we follow the convention of $(t.f)(x)=f(t^{-1}.x)$ for the action on the functions.

\begin{remark}
  \label{rem:birational}
  Consider the open subset $Y^\circ:=Y \setminus \supp \D \subset Y$. Then by construction $\pi^{-1}(Y^\circ)$ is $\TT$-equivariantly isomorphic to
  $X_\sigma \times Y^\circ$, where $\sigma = \tail \D$ and $X_\sigma$ is the toric variety associated to the polyhedral cone $\sigma$. Indeed, for the restriction
  $\mathcal{A}(\D)|_{Y^\circ}$ of the sheaf of algebras in  (\ref{eq:AH-sheaf}) we obtain
  \[\mathcal{A}(\D)|_{Y^\circ} = \bigoplus_{u \in \sigma^\vee \cap M} \CO_{Y^\circ}(\D(u)|_{Y^\circ}) = \bigoplus_{u \in \sigma^\vee \cap M} \CO_{Y^\circ},\]
  because $\supp \D(u) \subset \supp \D$ for every $u \in \sigma^\vee$ by definition and therefore $\D(u)|_{Y^\circ}=0$.
\end{remark}

The following theorem is the main result of~\cite{AH06}.
\begin{theorem}
  \label{thm:AH-main-result}
Let $X$ be a an affine $\mathbb{T}$-variety of complexity $c$. Then $X$ 
is $\mathbb{T}$-equivariantly isomorphic to $X(\D)$ for some proper polyhedral divisor $\D$ living on a semi-projective variety $Y$ of dimension $c$.

Moreover,  we have a commutative diagram
\[  \begin{tikzcd}
    \widetilde{X} \arrow[r,"r"] \arrow[d,"\pi"]& X \arrow[dl,"q", dashed]\\
      Y &
      \end{tikzcd}\]
with a good quotient $\pi \colon \widetilde{X}(\D)\rightarrow Y$, an equivariant proper birational morphism $r\colon \widetilde{X}(\D)\rightarrow X(\D)$ and $q\colon X \dashrightarrow Y$ being the rational quotient map induced by $\CC(Y) \cong \CC(X)^\TT$.
\end{theorem}

\section{Special degenerations and the multigraded Hilbert scheme}\label{sec:multi}
Our main goal here is to use Theorem~\ref{thm:zhuan-main} to obtain the boundedness result for K-semistable log Fano cone singularities. In order to get from the \emph{boundedness up to special degenerations} from Theorem~\ref{thm:zhuan-main} to the desired proper boundedness we need to first bound all cone $\TT$-singularities having a $\TT$-invariant degeneration to a given $\hat \TT$-singularity $X_0$. As it turns out for a fixed subtorus $\TT \subset \hat \TT$ a bounding family is given by the universal family of the so-called \emph{multigraded Hilbert scheme $\Hilb$}. We will see that every special $\TT$-invariant degeneration to $X_0$ induces a morphism $\mathbb A^1 \to \Hilb$. In this section we define some combinatorial gadgets associated to such a morphism.

In a second step we will then need show later that number possible subtori $\TT \subset \hat \TT$ for $\TT$-invariant special degeneration of K-semistable log Fano cone $\TT$-singularities (with normalized volume $> \epsilon$) is in fact finite. This will achieved in Section~\ref{sec:torus-inv-fam}.

\medskip
Consider some  cone $\TT$-singularity $X$ and let $M$ be the character lattice of $\TT$. Assume $X_0=\Spec A$ with $A \cong R/I$, $R=\CC[x_1 \ldots x_N]$ and $I \subset R$ being a homoegenous ideal with respect to an $M$-grading on $S$.

Then by the condition of being a cone singularity we have $S_0=\CC$ and the homogeneous components of $A$ are finite-dimensional as $\CC$-vector spaces. Hence, by \cite[Thm~1.1,1.2]{zbMATH02148302} there is a projective scheme $\Hilb:=\Hilb_{\mathbb A^m}$ parametrizing all $X \subset \mathbb A^m$ given by homogeneous ideals $J \subset R$ such that $\dim_\CC (R/J)_u = h(u):=\dim_\CC (R/I)_u$ for all $u \in M$. More precisely, $\Hilb$ represents the \emph{Hilbert functor}, which assign to every $\CC$-algebra $B$ the set of $\TT$-invariant subvarities of $\mathbb A^m$ corresponding to ideals $J \subset B \otimes R$ with $(B \otimes R_u)/I_u$ being locally free of rank $h(u)$. Hence, there exists a \emph{universal family}  $\mathcal F \subset \mathbb A^m \times \Hilb$, such that every $\TT$-invariant flat family $\mathcal F' \subset \mathbb A^N \times  S$ of such subvarieties can be obtained  via some morphism $\varphi \colon S \to \Hilb$ as $\mathcal F = \mathcal F \times_{\varphi} S$.  Moreover,  every $\hat \TT$-action on $\mathbb A^m$ induces a $\hat \TT$-action on $\Hilb$ and families $\mathcal F' \subset \mathbb A^m \times  S$, which are equivariant with respect to $\hat \TT$ will be induced by an equivariant morphism $S \to \Hilb$. The latter is a consequence of the uniqueness of the morphism $S \to \Hilb$ corresponding to $\mathcal F'$.

Also $\Hilb$ admits a $\hat \TT$-equivariant embedding into projective space and, therefore, a covering by $\hat \TT$-invariant affine charts.

We extend the above construction to cover the case of pairs $(X_0,B_0)$ consisting of a cone $\TT$-singularity and an integral $\TT$-invariant divisor as follows. Let $\mathcal F^{X_0}$ the universal family over $\Hilb$ and $\mathcal F^{B_0}$ the universal family of $\Hilbs^{B_0,\TT}$ (here we treat $B_0$ as a subscheme of $\mathbb A^m$). Then let $\Hilbp$ be the closed subscheme of $\Hilb \times \Hilbs^{B_0,\TT}$ cut out by the condition $B \subset X$ for $(X,B) \in \Hilb \times \Hilbs^{B_0,\TT}$. Then  $\Hilbp$ represents the corresponding functor for the case of pairs, where the universal family is given by 
\[\big((\mathcal F^{X_0}_X \times \Hilbs^{B_0,\TT})|_{\Hilbp},\quad  (\mathcal F^{B_0}_X \times \Hilbs^{X_0,\TT})|_{\Hilbp}\big).\]

Consider a cone singularity $X_0$ and an equivariant (non-product) special degeneration of $X$ with central fibre $X_0$. Such an special degeneral induces an action of $\TT \times \CC^*$ on $X_0$. Then we have $\TT \times \CC^* \subset \hat \TT$ for some maximal torus $\hat \TT$. Accordingly we have a surjection $\hat M \twoheadrightarrow M$ of the corresponding character lattices and an inclusion $N = \ker(\hat M \twoheadrightarrow M)^* \hookrightarrow \hat N$ of the co-character lattices. Moreover, the $\CC^*$-action of the special degeneration will induce an element $\hat v \in \hat N \setminus N$. Let $\hat \sigma \subset \hat N_\RR$ be the Reeb cone of $X_0$. The condition that $X$ is a cone singularity with respect to the $\TT$-action implies that $N_\RR$ intersects the interior of $\hat \sigma$ or equivalently $u \not\in (\hat \sigma^\vee \cup - \hat\sigma^\vee)$ for any non-zero element $u \in \ker (\hat M_\RR \twoheadrightarrow M_\RR)$.

\begin{lemma}
  \label{lem:graded-embedding-dim}
  Let $A$ be a positively (or negatively) graded $A_0$-algebra. We consider the ideal $I_+=\bigoplus_{\ell>0} A_\ell$ (or $I_+=\bigoplus_{\ell<0} A_\ell$, respectively). If $g_1,\ldots, g_m$ is a minimal set of generators of $A$ as a $A_0$ algebra, then the images form a minimal set of generators of $I_+/I_+^2$ as a $(A/I_+)$-module.
\end{lemma}
\begin{proof}
  Assume the set of the images of the $g_i$ is not minimal as a generating set. Then w.l.o.g. $g_1$ can be expressed as $g_1 = \sum_{j \geq 2} b^0_j g_j + f$, where $f \in I_+^2$ and $b^0_j \in A_0$.

  Since, the $g_1, \ldots, g_m$ generate $A$ as a $A_0$-algebra we may write $f$ as a polynomial in $g_1, \ldots, g_m$ with coefficients in $A_0$. Now for $\alpha \in \ZZ_{\geq 0}^m$ we set $g^\alpha=g_1^{\alpha_1} \cdots g_m^{\alpha_m}$ and $|\alpha|=\sum_j \alpha_j$ and we write
  \[
    f= \sum_\alpha b^0_\alpha g^{\alpha}.\]
  After restricting to the homogeneous part of degree $\deg(g_1)$ of the equation $g_1 = \sum_{j \geq 2}b^0_j g_j + \sum_\alpha b^0_\alpha g^\alpha$, we may assume that
$b_j^0 = 0$ for $\deg(g_j) \neq \deg(g_1)$ and $b^0_\alpha = 0$ for $\deg(g^\alpha) \neq \deg(g_1)$.   From the fact that $f \in I_+^2$ it follows that $|\alpha| > 1$ for $b^0_\alpha \neq 0$. Now, if $\alpha_1 \neq 0$ and $|\alpha| > 1$, then $\deg(g^{\alpha}) > \deg(g_1)$, since the degrees of the $g_i$ are all positive. Therefore, the coefficient $b^0_\alpha$ must be $0$ in this case. Hence,
  \[
    g_1 = \sum_{j\geq 2} b^0_jg_j + \sum_\alpha b^0_\alpha g^{\alpha}
  \]
  expresses $g_1$ as a polynomial in $g_2,\ldots,g_m$ with coefficients in $A_0$, but this contradicts the minimality of $\{g_1,\ldots, g_m\}$ as a set of $A_0$-algebra generators for $A$. 
\end{proof}

\begin{lemma}
  \label{lem:embedding-family}
  Given a cone $\TT$-singularity $X_0$, whereas for once we do not require normality, irreducibility or reducedness. Then there exists an affine space $\mathbb A^m$, such that every $\TT$-invariant (and $\CC^*$-equivariant) family $\X \to \Spec B$  over a (graded) local ring $(B,\m)$ with special fibre $X_0$ can be equivariantly embedded into $\mathbb A^m \times \Spec B$.

  Moreover, $\mathbb A^m$ can be equivariantly identified with the Zariski tangent space of $X_0$ at the vertex.
\end{lemma}
\begin{proof}
  Assume $\X=\Spec A \to \Spec B$. Then for an appropriate homomorphism $M \to \ZZ$ the $B$-algebra $A$ is positively graded with $A_0=B$ and there is a minimal set of homogeneous generators $g_1 \ldots g_m$. By Lemma~\ref{lem:graded-embedding-dim} their images give a minimal generating set of $I_+/I^2_+$. Now, it follows from Nakayama's Lemma (or its graded version, respectively) that the images of the of generating set form a basis $\bar{g}_1, \ldots,\bar g_m$ of $T^*=(I_+/I^2_+) \otimes_B (B/\m)$, the Zariski cotangent space of $X_0$ at the vertex $0$. Hence,
  we have an equivariant embedding $\X \hookrightarrow T \times \Spec B$, induced by the $B$-algebra homomorphism
  \[B \otimes_\CC \operatorname{Sym}(T^*) \to A;\quad  \bar g_i \mapsto g_i\]
\end{proof}

\begin{lemma}
  \label{lem:embedding-deg}
  Given a cone $\TT$-singularity $X_0$. There exists an affine space $\mathbb A^m$, such that every torus invariant special degeneration $\X \to \mathbb A^1$ to $X_0$ is equivariantly embedded as $\X \subset \mathbb A^m \times \mathbb A^1$.
\end{lemma}
\begin{proof}
  This follows from Lemma~\ref{lem:embedding-family} when taking $B$ to be the graded local ring $\CC[t]$.
\end{proof}

Hence, for $X_0=\Spec A$ and $A=R/I$ every $\TT$-cone singularity $X$ with $\TT$-invariant special degeration to $X_0$ will be represented by an element of the universal family over $\Hilbs=\Hilb$. Moreover, the special degeneration $\X \to \mathbb A^1$ induces a $(\TT \times \CC^*)$-equivariant map $\psi_\X \colon (\mathbb A^1,0) \to (\Hilbs, X_0)$ of pointed schemes.

We are now interested in studying the $\hat \TT$-action in a neighbourhood of $X_0$ in $\Hilbs$. For this we consider the tangent space $T_{X_0}(\Hilbs)$. By \cite[Prop.~1.6]{zbMATH02148302} this tangent space is isomorphic to
\(
 \Hom_R(I,R/I)^{\TT}.
 \)
 In the following we set $T^1_{X_0}=\Hom_R(I,R/I)$ and  $T^1_{X_0}(u)=\Hom_R(I,R/I)_u$ for $u\in \hat M$. Indeed, since $R/I$ was even $\hat M$-graded we obtain an $\hat M$-grading on $T^1_{X_0}$ as well. With this notation we have
 \[T_{X_0}\Hilbs = \bigoplus_{u \in \ker(\hat M \to M)}T^1_{X_0}(u).\]
 Now, the $\CC^*$-action of the special degeneration $\X$ acts with weight $\langle u, \hat v \rangle$ on $T^1_{X_0}(u)$.
 We consider a $\hat \TT$-invariant affine chart $U$ of $\Hilbs$ containing $X_0$. Let $C$ be the scheme-theoretic image of $\psi$ in $U$. Then $\TX:=T_{X_0}C \subset T_{X_0}\Hilbs$ is a $\CC^*$-invariant linear subspace.
 \begin{definition}   
   Consider an element $x\in W$ of some $\hat \TT$-representation $W$. We have a decomposition $x=\sum_{u\in \hat M}x_u$
   into homogeneous components $x_u$ of weight $u$. Then
   \[\supp x = \{u \in \hat M \mid x_u \neq 0\}\]
   is called the \emph{support} of $x \in W$. More generally for
   a subset $W'\subset W$ we set \[\supp W'= \bigcup_{x \in V'} \supp(x).\]
   
   For a special degeneration $\X$ with special fibre $X_0$ we define the \emph{degeneration cone of $\X$} by
   \[
     \dC := \operatorname{pos}(\supp \TX)^\vee:=\left(\sum_{u\in \supp \TX} \RR_{\geq 0}u\right)^{\bigvee},
   \]
   i.e. $\dC$ is the dual cone of the cone generated by the support of $\TX$.
 \end{definition}

\begin{example}\label{ex:running}
 We consider the toric $3$-fold hypersurface singularity
 \[
 X_0:=\{ (x,y,z,w) \mid xy+z^2=0\}. 
\]
This is the toric variety  $A_1 \times \mathbb A^1$, which corresponds to the cone $\hat \sigma$ spanned by the elements $(0,0,1),(0,2,1),(1,0,0)$. The $3$-dimensional torus $\hat{\mathbb{T}}$ acts on $X_0$ via
\[
(\lambda_1,\lambda_2,\lambda_2)\cdot (x,y,z,w) = 
(\lambda_2^{-1}x,\lambda_2\lambda_3^{-2}y,\lambda_3^{-1}z,\lambda_1^{-1}w),
\]
i.e. with weight matrix $$
\left(\begin{matrix}
    0&0&0&-1\\
    -1&1&0&0\\
    0&-2&-1&0
\end{matrix}\right).
$$
 For each $n \geq 0$, we have a special degeneration to $X_0$ given by
 \begin{equation}
 \X_{n}:=\{ (x,y,z,w,t) \mid xy+z^2+tw^{2k}\}.\label{eq:ex-degen}
\end{equation}
 This degeneration is equivariant with respect to the 2-dimensional subtorus $\TT=\{(\lambda_1^2,\lambda_2,\lambda_1^n) \in \hat{\TT}\}$. This corresponds to the surjective linear map $\RR^3=\hat M_\RR \to M_\RR= \RR^2$ with kernel $\RR\cdot(n,0,-2)$.

 Set $R=\CC[x,y,z,w]$ and $f=xy+z^2$. Then $T_{X_0}\Hilb = \Hom_R((f),R/(f))^\TT$ is spanned by the homomorphisms
sending $f$ to the class of monomials of the same $\TT$-weight as $f=xy+z^2$. In this case these are the monomials $z^2$ and $w^n$. Hence, $T_{X_0}\Hilbs = T^1_{X_0}(0,0,0) \oplus T^1_{X_0}(n,0,-2)$ with $T^1_{X_0}(0,0,0)=\CC\cdot (f \mapsto z^2)$ and $T^1_{X_0}(n,0,-2)=\CC\cdot (f \mapsto w^{n})$. We claim that $T\X=T^1_{X_0}(n,0,-2)=\CC\cdot (f \mapsto w^{n})$ and therefore
 $\dC = \left(\pos{(n,0,-2)}\right)^\vee$.
\end{example}

The following lemma shows that the elements $\Sigma_\X \cap \hat N$ parametrise the $\CC^*$-actions on $X_0$ which are induced by $\TT$-invaraint special degenerations of the generic fibre of $\X$ to $X_0$.
 \begin{lemma}
   \label{lem:support-and-deg-cone}
   Assume we have a special degeneration $\X$ of a $\TT$-singularity $X$ to a $\hat \TT$-singularity $X_0$ as above. Let the induced $\CC^*$-action on $X_0$ be given by an element $\hat v \in \hat N$. Then the following holds.
   \begin{enumerate}
   \item The weights in $\supp \TX$ span $\ker (\hat M_\RR \to M_\RR)$ as a vector
     space, \label{item:weight-span}
   \item we have $\hat v \in \rint \dC$ and \label{item:deg-cone}
   \item for every lattice element $\hat w \in \hat \sigma \cap \dC$ there exists a special degeneration $\X'$ of $X$ to $X_0$ inducing the $\CC^*$-action $\hat w$ on $X_0$. \label{item:deg-cone-2}
 \end{enumerate}
 \end{lemma}
 \begin{proof}
   Recall that $X_0$ was a $\hat\TT$-fixed point of $\Hilbs$. We may choose a  $\hat \TT$-invariant  open affine neighborhood $U=\Spec B$ of $X_0$ in $\Hilbs$.
   Then $\psi=\psi_\X$ corresponds to a $\ZZ$-graded homomorphism $\psi^\# \colon B \to \CC[t]$.
   The scheme-theoretic image of $\psi$ in $U$ is given by $C=\Spec(B/\ker \psi^\#)$.  By the definition of a special degeneration, $\CC^*$ acts with weight $1$ on the base of $\X$. It will therefore act with positive weight on $T_{X_0}C$. Indeed, on the level of coordinate rings we have $\deg(t)=-1$. Hence, $\CC[t]$ is negatively graded with respect to $\hat v$. Then $B/\ker \psi^\#$  will be negatively graded as well.  Hence, $\CC^*$ also acts via $\hat v$ with negative weights on the cotangent space and therefore with positive weights on the tangent space. From this we deduce that $\langle u, \hat v \rangle > 0$ for $u \in \supp T_{X_0}C$ and therefore $\hat v \in \rint \dC$ as claimed in (\ref{item:deg-cone}).
 
By Lemma~\ref{lem:graded-embedding-dim} we have  an equivariant embedding
\[C \hookrightarrow T_{X_0}C \subset T_{X_0}\Hilbs.\]
We have $\supp T_{X_0}C \subset \supp T_{X_0}\Hilbs \subset \ker (\hat M_\RR \to M_\RR)$.
Consider now the image $\psi(1) \in C \hookrightarrow  T_{X_0}C \subset  T_{X,0}\Hilbs$. Then we have
$\supp(\psi(1)) \subset \supp T_{X_0}C$. Now, $\psi(1)$ is fixed under the subtorus $\TT'$ corresponding
to the surjection $\hat M \to \hat M/\lspan_\RR(\supp\psi(1))$. Since $\psi(1)$ parametrizes $\X_1\cong X$ in the Hilbert scheme $\Hilbs$ it follows that $X$ admits an action of this subtorus. Because $X$ was assumed to be a $\TT$-variety in our strict sense, it follows that $\TT=\TT'$ and $\lspan_\RR(\supp\psi(1)) = \ker (\hat M_\RR \to M_\RR)$. But then also $\lspan_\RR(\supp T_{X_0}C) = \ker (\hat M_\RR \to M_\RR)$ holds. This shows (\ref{item:weight-span}).

To prove (\ref{item:deg-cone-2}) we need to refine the ideas from above a little bit. We consider the $\hat T$-invariant subspace
   \[W:=(\ker \psi^\#) \cap B_0 \oplus \bigoplus_{\langle u, \hat v \rangle >0} B_u \subset B.\]
   Then $W \subset \ker \psi^\#$ holds, since $\CC[t]$ was negatively graded. Let $I \subset B$ be the $\hat M$-graded ideal generated by $W$. Then $\psi \colon \mathbb A^1 \to \Spec B$ factors through a $\hat T$-equivariant closed embedding $S:=\Spec B/I \hookrightarrow \Spec B$. Note, that $B/I$ is negatively graded by construction and $X_0 \in S$ is the unique fixed point of the corresponding $\CC^*$-action. Hence, by Lemma~\ref{lem:graded-embedding-dim} we have  a $\hat T$-equivariant embedding
   \[S\hookrightarrow T_{X_0}S \subset T_{X_0}\Hilbs\]
   which maps $X_0 \in S$ to $0 \in T_{X_0}S$.  Now, we have \[\psi(1) \in S \subset T_{X_0}S \subset T_{X_0}\Hilbs\] and for every $\hat w \in \rint \dC$ the corresponding $\CC^*$-action has positive weights along the homogeneous components of $\psi(1)$ in $T_{X_0}S$. Hence, in $T_{X_0}S$ we have
   \[\lim_{t\to 0} \gamma_{\hat w}(t).\psi(1) = 0,\] which corresponded to $X_0 \in S \subset \Hilbs$. In this way we obtain a special degeneration of $X$ to $X_0$ and inducing the $\CC^*$-action given by $\hat w$.
\end{proof}
\begin{example}\label{ex:running2}
  We are coming back to Example~\ref{ex:running}. We choose now a $\CC^*$-action on the total space of the degeneration $\X_{n}$ given in (\ref{eq:ex-degen}) by acting via $\lambda.(x,y,z,w,t)=(x,\lambda^2y,\lambda z, w, \lambda^2 t)$. The corresponding one-parameter subgroup in $\hat N$ is given by $\hat v=(0,0,-1)$. Hence, we see indeed that $\hat v \in \dC$.
\end{example}

Assume now, we are given a $\hat \TT$-invariant family $\X \subset \mathbb A^m \times S \to S$ of cone singularities  of arbitrary complexity with common Reeb cone $\hat \sigma \subset \hat N_\RR$.

Note, that for rays in a real vector space we may talk about distances, accumulation and convergence by fixing an inner product and intersecting with a unit sphere.

\begin{lemma}
  \label{lem:accumulation}
  Given a $\hat \TT$-invariant family $\X \to S$ of cone singularities  of arbitrary complexity with common Reeb cone $\hat \sigma \subset \hat N_\RR$.
  We consider the rays in $\hat M_\RR$ spanned by degrees $u \in \hat M \setminus  \hat \sigma^\vee$, such that $T^1_{\X_s}(u) \neq 0$ for at least one $s \in S$. Then those rays accumulate only at  the boundary of $\hat \sigma^\vee$.  In particular, we obtain the following.
  \begin{enumerate}
  \item For fixed $\xi_0 \in \rint \hat \sigma$ we have
    $\langle u, \xi_0 \rangle \leq 0$ only for fninitely many of those
    rays.\label{item:accumulation-finite}
  \item Assume for every such ray  we have a $\xi_u \in \rint \hat \sigma$ with $\langle u, \xi_u \rangle \leq 0$. Then $\xi_u/|\xi_u|$ accumulate at faces of $\hat \sigma$ which are dual to a faces of $\hat \sigma^\vee$ that contain an accumulation point of the $u/|u|$ in their interior.\label{item:dual-accumulation}
\end{enumerate}

\end{lemma}
\begin{proof}
  We may assume wlog that $S=\Spec B$ is affine and
  \(\X=\Spec \mathcal A\)
  with $\A=(B \otimes R)/\I$ and the ideal $\I=(f_1,\ldots,f_r)$ being generated by homogeneous elements of weight $\deg f_i=: w_i \in \hat M$. The fibre over $s \in S$ is then given by
  $\X_s = \Spec \A_s$ with $\A_s = \A \otimes \CC(s) = R/\I_s$. Note, that $\I_s = \I \otimes \CC(s)$ is generated by elements in degrees $w_1, \ldots, w_r$, as well.

  We consider the non-trivial homogeneous components of  $T^1_{\X_s}=\Hom_R(\I_s, R/\I_s)$. Let $\varphi \in \Hom_R(\I_s, R/\I_s)$ be a non-trivial element. Then we have $0 \neq \varphi(f_i) \in (R/\I_s)_w$ for at least one $i \in \{1,\ldots,r\}$ and some $w \in \hat \sigma^\vee$. In this case we have $\deg \varphi=w-w_i$. Hence, we get $T^1_{\X_s}(u) \neq 0$ only for  $u \in \left(G + \hat \sigma^\vee\right)$, where $G=\{-w_1, \ldots,-w_r\}$.

  Consider a sequence of pairwise distinct elements $(u_i)_{i \in \NN}\in \hat M \setminus  \hat \sigma^\vee$ with $T^1_{\X_s}(u) \neq 0$ for some $s \in S$.   Then we have $\sfrac{u_i}{|u_i|} \in (\sfrac{G}{|u_i|} + \hat \sigma^\vee) \setminus \hat \sigma^\vee$. Since 
  the $u_i$ are lattice points we necessarily have  $|u_i| \to +\infty$. We see that the distance between $\sfrac{u_i}{|u_i|}$ and $\hat \sigma^\vee$ goes to $0$. Hence, every accumulation point has to lie on the boundary of $\hat \sigma^\vee$.

  Now, we show (\ref{item:dual-accumulation}). Consider a sequence $u_i \to u \in \rint \tau \prec \hat\sigma^\vee$.
  Assuming $\langle\xi_{u_i}, u_i \rangle \leq 0$ we obtain
  \(\lim_{i\to \infty} \langle\sfrac{\xi_{u_i}}{|\xi_{u_i}|} , u \rangle = \lim_{i\to \infty} \langle\sfrac{\xi_{u_i}}{|\xi_{u_i}|} , u_i \rangle \leq 0.\)
  On the other hand, it follows from $u \in \hat\sigma^\vee$ that 
  $\langle\sfrac{\xi_{u_i}}{|\xi_{u_i}|} , u \rangle \geq 0$ and therefore $\lim_{i\to \infty} \langle\sfrac{\xi_{u_i}}{|\xi_{u_i}|} , u \rangle = 0$. In other words, the $\xi_{u_i}$ accumulate at $\tau^*=u^\perp \cap \hat \sigma$.

  The statement (\ref{item:accumulation-finite}) follows from (\ref{item:dual-accumulation}), since the finite set $\{\xi_0\} \subset \rint \hat\sigma$ cannot have accumulation points at the boundary of $\hat\sigma$.
\end{proof}

\begin{example}
  Let us once again consider the degenerations from Example~\ref{ex:running}. We had seen that non-zero graded components $T^1_{X_0}(u)$ appearing here have degree $u=(n,0,-2)$. Here, $u/|u|$ converges to $(1,0,0)$, which is indeed a boundary point of the Reeb cone $\hat \sigma$.
\end{example}

\section{Torus invariant families}
\label{sec:torus-inv-fam}
In this section we study torus invariant families of cone singularities and special degenerations to elements of such a family, but first we show that we can always reduce to the case that a bounding family for a class
of $\TT$-singularities is $\TT$-invariant.

\begin{proposition}
  \label{prop:equivariantly-bounded}
Let $\{(Y_i,A_i)\}$ be log bounded family of pairs where $Y_i$ is a projective variety and $A_i$ is an ample $\qq$-divisor.
Let $\{(X_i;x_i)\}$ be the corresponding family of $\mathbb{T}$-singularities.
Then, we can find an affine space
\[
\mathbb{A}_{x_1,\dots,x_n}\times \mathbb{A}_{t_1,\dots,t_k}, 
\]
a family of ideals $I_{\mathbf t} =\langle f_1(\mathbf x,\mathbf t),\dots,f_s(\mathbf x,\mathbf t)\rangle$
depending on $\mathbf t\in Z\subset  \mathbb{A}_{t_1,\dots,t_k},$ and an embedding  $\TT \hookrightarrow \operatorname{Aut}(\mathbb{A}_{x_1,\dots,x_n})$, such that
for some $\mathbf{t}_i \in \mathbb{A}_{t_1,\dots,t_k}$ the vanishing locus $V(I_{\mathbf t_i})$ is $\TT$-invariant and $(X_i;x_i)$ is $\TT$-equivariantly isomorphic to $(V(I_{\mathbf t_i});0)$.
\end{proposition}

\begin{proof}
Let $\mathcal{Y}\rightarrow T$ be the bounding family for the projective varieties $Y_i$ and $\mathcal{A}$ be the bounding divisor for $A_i$. 
We may assume that $\mathcal{A}$ is ample over $T$ (see, e.g.,~\cite[Theorem 1.2.13]{Laz04a}).
Then, $\mathcal{A}^{\otimes m}$ is $T$-very ample for some $m$ only depending on the pair $(\mathcal{Y},\mathcal{A})$.
We may further assume that if $\mathcal{Y}_t$ admits a $\mathbb{T}$-action, then $\mathcal{A}_t^{\otimes m}$ admits a $\mathbb{T}$-linearlization.
Hence, we may embedd all fibers of $\mathcal{Y}\rightarrow T$ to $\mathbb{P}^N$ as varieties of degree at most $d$.
We call this embedding $\phi_t$.
Here, both $N$ and $d$ only depend on the bounding family.
Furthermore, if $\mathcal{Y}_t$ admits a $\mathbb{T}$-action, then
this embedding is $\mathbb{T}$-equivariant.
Let $\mathbb{T}_0\leqslant {\rm PGL}_N(\kk)$ be a maximal torus.
For every $\mathcal{Y}_t$ that admits a torus action, 
we can find $g\in {\rm PGL}_N(\kk)$ such that $\mathcal{Y}_t$ is equivariantly
isomorphic to $g^{-1}\phi_t(\mathcal{Y}_t)g$ with the action endowed
by a subtorus of $\mathbb{T}_0$.
Since the degree of $g^{-1}\phi_t(\mathcal{Y}_t)g$ is $d$ and it is a non-degenerate variety, then 
there are only finitely many tori
$\mathbb{T}_1,\dots,\mathbb{T}_k \leqslant \mathbb{T}_0$ that can act on it.
Then, the disjoint union 
$\bigcup_{i=1}^k {\rm Hilb}_{d,N}^{\mathbb{T}_i}$
of $\mathbb{T}_i$-invariant points of the Hilbert scheme
gives an equivariant bounding family for the pairs $(Y_i,A_i)$.
By taking the cone of the $Y_i$ with respect to the $\qq$-polarization $A_i$ on this family, we obtain the desired result.
\end{proof}


The aim of the next technical lemma is to identify a convenient open subset $S'\subset S$ in the base of a $\TT$-invariant family $\X \to S$ of $\TT$-cone singularities, such that for every primitve lattice element $\xi$ in the interior of the Reeb cone we are able to extract a prime divisor $D_\xi$ over $\X$, which restricts to a prime divisor in every fibre over elements $s \in S'$. Moreover, both $D_\xi$ and $D_\xi|_s$ induce the valutation corresponding to $\xi$ on $\X$ and $\X_s$, respectively. 
\begin{lemma}
  \label{lem:extracted-reeb-divisor}
  Let $f \colon \X \to S$ be a $\TT$-invariant family of $\TT$-cone singularities over a normal base $S$.
  Then there exist an open subset $S' \subset S$, such that for every rational ray $\xi$ in the interior of the common Reeb cone $\sigma$ of $\X$ there exist
  \begin{itemize}
  \item a proper birational morphism $\psi \colon \X' \to \X$,
  \item a surjective morphism $U \to S'$ and
  \item a $\TT$-invariant open subset $V'\subset \X'$ equivariantly isomorphic to $X_\xi \times U$, where
    $X_\xi$ is the affine toric variety corresponding to the polyhedral cone $\xi$,
  \end{itemize}
  such that the following $\TT$-equivariant diagram commutes
    \[  \begin{tikzcd}
        V'\arrow[d,hook] & \arrow[l,"\cong"'] X_\xi \times U  \arrow[r,"p_2"] & U \arrow[d]\\
        \X' \arrow[r,"\psi"] & \X \arrow[r,"f"]& S'
      \end{tikzcd}\]
and $D_\xi \times U \subset X_\xi \times U$ is the only component of the exceptional locus of $\psi$ intersecting $V'$.
\end{lemma}
\begin{proof} After passing to an open affine subset of $S$ we may assume that $S$ and $\X$ are affine.
  Since $\X \to S$ is a $\TT$-invariant family of $\TT$-cone singularities, we have $\CC[\X]^\TT = \CC[S]$ and the structure morphism of the family is induced by the inclusion of algebras.

  It follows from Theorem~\ref{thm:AH-main-result} that there exists a semi-projective variety $\Y$, a proper $\TT$-equivariant morphism $r \colon \widetilde{\X} \to \X$ and good quotient morphisms $\pi \colon \widetilde{\X} \to \Y$. We obtain $H^0(\CO_\Y)=H^0(\CO_{\widetilde{\X}})^\TT=H^0(\CO_\X)^\TT$. In particular,
  $\Spec(H^0(\CO_\Y))=S$ holds and $\Y$ is projective over $S$. Altogether we have a commutative diagram
  \[  \begin{tikzcd}
      \widetilde{\X} \arrow[d,"\pi"'] \arrow[r,"r"]&  \X\arrow[d,"f"]  =\Spec(H^0(\widetilde{\X}, \CO_{\widetilde{\X}}))\\
       \Y \arrow[r,"p"]& S = \Spec(H^0(\Y,\CO_\Y)).
     \end{tikzcd}\]
   Let $X_\sigma$ be the toric variety corresponding to the Reeb cone $\sigma$.   By Remark~\ref{rem:birational} there is an open subset $\Y^\circ \subset \Y$, such that   $X_\sigma \times \Y^\circ$ is $\TT$-equivariantly isomorphic to an open subset  $V \subset \widetilde{X}$. Let $U=\Y^\circ \setminus \overline{p(E^{\text{vert}})}$. Here, $E^{\text{vert}}$ denotes the union of the vertical components of the exceptional locus
  of $r \colon \widetilde{\X} \to \X$. By vertical components we mean those whose image under $p$ has positive codimension. We then set $S'=p(U)$.
  
  Let $X_\Sigma$ be the  toric variety corresponding the fan $\Sigma$ obtained by star subdividing $\sigma$ with respect to $\xi$. Let $\bar X_\Sigma$ be some $\TT$-equivariant completion of $X_\Sigma$.  There is a $\TT$-equvariant birational map
  $\varphi \colon X_\Sigma \times \Y \dashrightarrow \widetilde{\X}$ and we obtain $\X'$ as a resolution of indetermacy by taking the closure of the graph of $\varphi$ in $(X_\Sigma \times \Y) \times_S \widetilde{\X}$. This results in the following commutative diagram.
\[  \begin{tikzcd}
\X' \arrow[r,hook]  &   (\bar X_\Sigma \times \Y) \times_S \widetilde{\X} \arrow[dl,"p_1"'] \arrow[d,"p_2"]\\
 \bar X_\Sigma \times \Y   \arrow[r,dashed]  & \widetilde{\X}
\end{tikzcd}\]
Here, $p_1$ and $p_2$ are the projections to the first and second factor, respectively.  Notice that $p_2$ is proper because $\bar X_\Sigma \times \Y$ is proper over $S$ and properness is preserved under base extension. Therefore, the induced map $\X' \to \widetilde{\X}$ is proper because it is the composition of a closed embedding with $p_2$. Hence, the composition $\psi \colon \X' \to \widetilde{X} \to \X$ is proper as well.

The rational map $\bar X_\Sigma \times \Y \dashrightarrow \widetilde{\X}$ is defined on the open subset $X_\Sigma \times U$. Indeed, here, it is induced simply by the toric morphism $X_\Sigma \to X_\sigma$. Therefore, over $X_\Sigma \times U$ the resolution of indetermacy $\X' \to \bar X_\Sigma \times \Y$ is an isomorphism and we may identify $X_\Sigma \times U$ with a $\TT$-invariant open subset of $\X'$.  We now set
\[V' \;:=\; X_\xi \times U   \;\subset\; \X'.\]
Then $D_\xi  \times U$ is the only component of the exceptional locus of $\psi \colon \X' \to \X$ which intersects $V'$. Indeed, $D_\xi\times U$ is the only \emph{horizontal} closed $\TT$-invariant subset in  $X_\xi \times U$, i.e. the only one which projects onto $U$. On the other hand, we had chosen $U$, such that $X_\Sigma \times U$ cannot contain any vertical component.
\end{proof}

\begin{proposition}
  \label{prop:minimizer-constant}
  Consider $\hat \TT$-invariant family $(\mathcal X, \Delta) \to S$ of log Fano cone singularities of dimension $d$. We denote by $\hat \sigma$ the common Reeb cone of the family.

  Then there exists a stratification of $S$ into finitely many locally closed subsets, such that for $\xi \in \hat \sigma$ the maps $s \mapsto a_{(\mathcal X_s,\Delta_s)}(\xi)$ and  $s \mapsto \nvol_{(\mathcal X_s,\Delta_s)}(\xi)$ are constant on the strata. In particular, the minimizer of $\nvol_{(\mathcal X_s,\Delta_s)}$ in $\hat \sigma$ is the same for every element $s$ of a fixed stratum.

  This stratification depends only on $\X \to S$, but not on $\Delta$.
\end{proposition}
\begin{proof}
  First, we may reduce to the case of a normal base $S$. This can be done by passing to the normalization of te underlying reduced scheme and considering the pullback of the family along the normalization. Then we may treat the resulting finitely many connected components separately. By induction on the dimension of $S$ it is sufficient to show that the maps are constant on a Zariski open subset of $S$. Hence, we may additionally assume that $S$ is affine, $\X \to S$ is flat and the total space is $\QQ$-Gorenstein.

  We may consider $\xi$ as a torus invariant valuation on $\X$. Moreover, using an approximation argument we may assume that $\xi$ is a primitive lattice element.
  
  We now consider the proper birational morphism $f \colon \X' \to \X$
  and the open subset $V' \cong X_\xi \times U$ from Lemma~\ref{lem:extracted-reeb-divisor}.
  We have  $V'_s= X_\xi \times U_s$ and $(D_\xi \times U)|_{V'_s} = D_\xi \times U_s$
  The product structure of $V'_s$ lets us conclude that, as in the toric setting, the divisorial valuation 
  corresponding to $D_\xi \times U_s$ coincides with the one induced by $\xi$.  We are now comparing the log discrepancy of $D_\xi \times U$ and $D_\xi \times U_s$ on  the open subsets $V'\subset \X'$ and $V'_s\subset \X'_s$, respectively.  We are interested in the coefficient of the divisor $D_s:=K_{X'_s}+\psi_*^{-1}\Delta_s - \psi^*(K_{\X_s}+\Delta_s)$ at $\overline{D_\xi \times U_s}$. Let us set $D:=K_{X'/S'}+\psi_*^{-1}\Delta - \psi^*(K_{\X/S'}+\Delta)$. Then we have   \(D_s =D|_{\X_s'}.\)
  Since $D$ is only supported in exceptional divisors we have $D|_{V'}=a_{(\X,\Delta)}(\overline{D_\xi \times U})\cdot(D_\xi \times U)$. It follows that
  \[ a_{(\X_s,\Delta_s)}(\overline{D_\xi \times U_s})\cdot(D_\xi \times U_s)=
    D|_{V'_s}=
    a_{(\X,\Delta)}(\overline{D_\xi \times U})\cdot(D_\xi \times U_s)\]
  Now, comparing coefficients gives
  \[
    a_{(\X_s,\Delta_s)}(\xi)=a_{(\X_s,\Delta_s)}(\overline{D_\xi \times U_s})=a_{(\X,\Delta)}(\overline{D_\xi \times U}).
  \]
  In particular, $a_{(\X_s,\Delta_s)}(\xi)$ does not depend on $s$.

  It now follows from flatness that $\vol_{\X_s}(\xi)$ and therefore also  $\nvol_{(\mathcal X_s,\Delta_s)}$ is constant along $S'$.
\end{proof}


\begin{proposition}
  \label{prop:ksstable-deg}
  Assume $(X,\Delta,\xi)$ is a K-semistable polarized log Fano cone $\TT$-singularity with $\TT$-invariant special degeneration $(\X,\mathbf{\Delta})$ to $(X_0,\Delta_0)$ with degeneration cone $\dC \subset \hat N_\RR$. Then we have
\[\nvol(X,\Delta) = \nvol_{(X_0,\Delta_0)}(\xi) \leq \nvol_{(X_0,\Delta_0)}(\xi')\]
for every element $\xi' \in \dC$.

Moreover, if $V \subset \hat N_\RR$ is any subspace with $N_\RR \subsetneq V$ and $\xi_0$ is the unique minimizer $\xi_0$ of the normalized volume among the valuation corresponding to elements of $\hat \sigma \cap V$, then $\xi_0$ is not contained in $\dC \setminus N_\RR$.
\end{proposition}
\begin{proof} By Lemma~\ref{lem:support-and-deg-cone} and using the approximation result of \cite[Lemma~2.10]{li2017stability} we can find a valuation $w$ on $X$, inducing a degeneration to $X_0$ and the polarization $\xi'$. We then have $\nvol_{(X_0,\Delta_0)}(\xi') = \nvol_{(X,\Delta)}(w) \geq \nvol(X,\Delta)$ by definition of the normalized volume. Moreover, we obtain $\nvol(X,\Delta)=\nvol_{(X,\Delta)}(\xi)$ by the K-semistability of $(X,\Delta)$ and
  $\nvol_{(X_0,\Delta_0)}(\xi) = \nvol_{(X,\Delta)}(\xi)$ by \cite[Lem.~3.2.]{li2017stability}. 

  Let now $\xi_0$ be the unique minimizer of $\nvol_{(X_0,\Delta_0)}$ on $\hat \sigma \cap V$ (existence and uniqueness follows from the convexity properties of the volume functional). Since  $\xi_0$ was a minimizer we have $\nvol_{(X_0,\Delta_0)}(\xi) \geq \nvol_{(X_0,\Delta_0)}(\xi_0)$. If we assume  $\xi_0 \in \dC\setminus N_\RR$, then
  as we have just seen, $\nvol_{(X_0,\Delta_0)}(\xi) \leq \nvol_{(X_0,\Delta_0)}(\xi_0)$ holds. Hence, we obtain
  $\nvol_{(X_0,\Delta_0)}(\xi) = \nvol_{(X_0,\Delta_0)}(\xi_0)$. By the uniqueness of the minimizer this implies $\xi_0=\xi$, but this is a contradiction to $\xi \in N_\RR$.
  
\end{proof}
\begin{example}
  \label{ex:running4}
  We continue with the toric variety $X_0$ and its deformation from Example~\ref{ex:running}. Since the cone $\hat \sigma$ is simplicial the minimizer for the normalized volume is the barycentric ray spanned by $\xi=(1,2,2)$. We see that starting from $k=5$ this ray is contained in the interior of $\Sigma_{\X_{2k}}=(\pos(n,0,-2))^\vee$. Hence, by Proposition~\ref{prop:ksstable-deg} those deformations are not longer K-semistable. This agrees with the findings from \cite{zbMATH06868031}.
\end{example}

Let $\X \to S$ be a $\hat \TT$-invariant flat family of cone singularities and $X_0$ any fibre. Given a fixed subtorus $\TT \subset \hat \TT$ the multigraded Hilbert scheme $\Hilb$ contains all $\TT$-varieties with $\TT$-invariant degenerations to an element of $\X$. In particular, the class of a all such cone singularities is bounded. However, there are, of course, infinitely many possible choices of a subtorus $\TT \subset \hat \TT$. Hence, to obtain a boundedness result, we need to limit these choices of subtori to a finite number. This is done in the following proposition for the case of an empty boundary. This will be superseded by Proposition~\ref{prop:finite-number-of-tori-pairs}, which covers the case of pairs. However, we decided to include the proposition and its proof anyway, since the proof for an empty boundary is much easier and it might therefore serve as a warmup and illustration of the general strategy.
\begin{proposition}
  \label{prop:finite-number-of-tori}
  Fix a $\hat \TT$-invariant family $\X \subset \mathbb A^N \times S \to S$ of Fano cone $\hat \TT$-singularities. Then there are only finitely many subtori $\TT \subset \hat \TT$, such that there exists a $\TT$-invariant special degenerations of K-semistable Fano cone $\TT$-singularities to an element $\X_s$ of $\X$.
\end{proposition}
\begin{proof} After passing to one of the strata in Proposition~\ref{prop:minimizer-constant} we may assume that the minimizer $\xi_0 \in \hat\sigma$ of the normalized volume on $\X_s$ is the same for every $s \in S$. 
  
Recall that the subtori $\TT \subset \hat \TT$ are in one-to-one correspondence with surjective lattice homomorphisms $\hat M \twoheadrightarrow M$ or equivalently to rational subspaces $K=\ker(\hat M_\RR \twoheadrightarrow M_\RR)$.

  Assume there is an infinite sequence of subtori $\TT_i$ corresponding to rational subspaces $K_i\subset \hat M$, with non-product $\TT$-invariant special degenerations $\Y_i$ of a K-semistable cone $\TT_i$-singularity to an element $\X_{s_i}$ of our family and  with $G_i:=\supp T\Y_i$. According to Lemma~\ref{lem:support-and-deg-cone} we have $K_i = \lspan G_i$ and by definition we have $\Sigma_{\Y_i} = (\sum_{u\in G_i}\RR_{\geq 0}\cdot u)^\vee$.

  Let $K$ be a subspace of maximal dimension (possibly $0$), which is contained in infinitely many of the $K_i$. We may pass to an infinite subsequence such that always $K \subset K_i$ holds. We then consider $G_i' = G_i \setminus K$. By the maximality of $K$, multiples of an element $u \in \hat M$ can only appear in finitely many of the $G_i'$.
  By applying Proposition~\ref{prop:ksstable-deg} for $V=K^\perp$ we see that $\langle u, \xi_0 \rangle \leq 0$ holds for at least one $u$ in each $G_i'$, but this leads to a contradiction, since by Lemma~\ref{lem:accumulation} this inequality can only hold for finitely many such $u$.
\end{proof}

\subsection{Log boundedness}
\label{sec:log-boundedness}
To generalize Proposition~\ref{prop:finite-number-of-tori} to the case of pairs, we will frequently consider the following situation.

There is a sequence of $\hat \TT$-invariant families $(\X,\boldsymbol{\Delta}_i) \to S$ of log Fano cone $\TT$-singularities  with
the same underlying family $\X \to S$ and common Reeb cone $\hat \sigma$, such that $\supp \boldsymbol{\Delta}_i \subset \mathcal{B}$ for all $i$ and some reduced divisor $\mathcal{B} \subset \X$. We then consider elements $(X_i,\Delta_i)$ of those families, i.e. $X_i = \X_{s_i}$ and $\Delta_i=(\boldsymbol{\Delta}_i)_{s_i}$ for some element $s\in S$. By \cite[Lemma 2.18.]{li2017stability} and Proposition~\ref{prop:minimizer-constant} for the elements $\xi \in \hat \sigma$ the log discrepancy $a_{(X,\Delta_i)}(\xi)$ is given by an element $a_i \in \hat\sigma^\vee$ via \[a_{(X_i,\Delta_i)}(\xi) = a_i(\xi)=\langle a_i, \xi \rangle\] and $a_i$ does not depend on the choice of $s_i \in S$ (possibly after passing to a stratum of the stratification in Proposition~\ref{prop:minimizer-constant}).

\begin{proposition}
  \label{prop:finite-number-of-tori-pairs}
  
  Fix $\epsilon > 0$ and a $\hat \TT$-invariant family $\X \subset \mathbb A^N \times S \to S$ of $\TT$-cone singularities with a reduced divisor $\mathcal{B} \subset \X$.  Then there are only finitely many subtori $\TT \subset \hat \TT$, such that there exists a $\TT$-invariant special degenerations of K-semistable log Fano cone $\TT$-singularities with normalized volume at least $\epsilon$ to a pair $(\X_s,\Delta_s)$ with $\X_s$ being an element of $\X$ and $\supp \Delta_s \subset \mathcal{B}_s$.

\end{proposition}
\begin{proof} After passing to one of the finitely many strata in Proposition~\ref{prop:minimizer-constant} we may assume that the the log discrepancy of an element $\xi \in \hat \sigma$ is constant on $S$.
  
   Recall that the subtori $\TT \subset \hat \TT$ are in one-to-one correspondence with surjective lattice homomorphisms $\hat M \twoheadrightarrow M$ or equivalently to rational subspaces $K=\ker(\hat M_\RR \twoheadrightarrow M_\RR)$.   Assume there is an infinite sequence of subtori $\TT_i$ corresponding to rational subspaces $K_i\subset \hat M$, with non-product $\TT$-invariant special degenerations $(\Y_i, \mathcal{D}_i)$ of a K-semistable cone $\TT_i$-singularity $(Y_i,D_i)$ to an element $(X_i,\Delta_i)$ of our family and with $\nvol(Y_i,D_i) > \epsilon$.
   We set $G_i:=\supp T\Y_i$. According to Lemma~\ref{lem:support-and-deg-cone} we have $K_i = \lspan G_i$ and by definition we have $\Sigma_{\Y_i} = (\sum_{u\in G_i}\RR_{\geq 0}\cdot u)^\vee$. Let $\xi_i$ be the minimizer of $\nvol_{(X_i,\Delta_i)}$ on $\hat \sigma$.

   We denote by $K$ the subspace of maximal dimension (possibly $0$), which is contained in infinitely many of the $K_i$. We may pass to an infinite subsequence such that always $K \subset K_i$ holds. We then consider $G_i' = G_i \setminus K$. By the maximality of $K$, multiples of an element $u \in \hat M$ can only appear in finitely many of the $G_i'$. Now, by Lemma~\ref{lem:accumulation} the rays spanned by those $u_i$ accumulate only at the boundary $\partial \hat \sigma^\vee$ of $\hat \sigma^\vee$. Hence, after passing to a subsequence we may assume that

   \[\dist(G_i',\hat \sigma^\vee):=\max \{\dist(\sfrac{u}{|u|}, \hat\sigma^\vee) \mid u \in G_i'\} \to 0\]
 for $i \to \infty$.
   
  By applying Proposition~\ref{prop:ksstable-deg} for $V=K^\perp$ we see that $\langle u_i, \xi_i \rangle \leq 0$ holds for at least one $u_i$ in each $G_i'$.  By the above, after passing to a subsequence we may assume that $u_i/|u_i| \to  u_\infty \in \rint \tau$, where $\tau \prec \hat\sigma^\vee$ is some face.  Then by Lemma~\ref{lem:accumulation}~(\ref{item:dual-accumulation}) the $\xi_i$ accumulate at the face $\tau:=u^\perp_\infty \cap \hat \sigma$. After passing to a subsequence again we may assume that $\xi_i$ converges to some $\xi_\infty \in \tau$.  By applying Lemma~\ref{lem:sequence-of-boundaries1} below we may assume that the sequence $a_i \in \hat \sigma^\vee$ corresponding to the log discrepancies with respect to $\Delta_i$ converges to some element $a_\infty \in \tau^*$.

  Pick any interior element $w$ of $K^\perp \cap \hat  \sigma$. Since the distance of $G_i'$ to  $\hat\sigma^\vee$ converges to $0$, by Lemma~\ref{lem:sequence-of-cones} below, we
  have that  starting from some index the element $w$ will be contained in all cones $(\sum_{u \in G_i'} \RR_{\geq 0}u)^\vee \subset \Sigma_{\Y_i} \cap K$. Now, take $v_i=\lambda_i w + (1-\lambda_i)\xi_\infty \in \Sigma_{\Y_i}$ with $\lambda_i \in [0,1]$ being minimal with this property (i.e. $v_i$ lies on the boundary of $\Sigma_{\Y_i}$). By Lemma~\ref{lem:sequence-of-boundaries2} below we have $\nvol_{(X_i,\Delta_i)}(v_i) \to 0$. On the other hand, Proposition~\ref{prop:ksstable-deg} implies $\nvol(Y_i,D_i) \leq \nvol_{(X_i,\Delta_i)}(v_i)$, which contradicts our assumption 
  $\nvol(Y_i,D_i) > \epsilon$.
\end{proof}

\begin{lemma}
  \label{lem:sequence-of-boundaries1}
  In the situation described at the beginning of Section~\ref{sec:log-boundedness} assume the minimizers $\xi_i$ of $\nvol_{(X_i,\Delta_i)}$ converge to an element $\xi_\infty$ of a face $\tau \prec \hat \sigma$. Then, possibly after passing to a subsequence, $a_{i}$ converges to an element of the dual face $\tau^*$.
\end{lemma}
\begin{proof}
  Pick any element $\xi_0$ in the interior of $ \hat \sigma$. Set \[a_0=\max \{ a_{(\X_{s},\boldsymbol{\Delta}_{s})}(\xi_0) \mid \boldsymbol{\Delta} \text{ is a $\RR$-Gorenstein boundary with } \supp(\boldsymbol{\Delta}) \subset \mathcal{B}\}.\]
Note, that the maximum exists as the $\RR$-Gorenstein boundaries with support in $\mathcal{B}$ form a compact set and the log discrepancy varies continuously with the boundary. Moreover, it does not depend on $s$.
  Then we have
  \[
    a_i(\xi_i)^n\vol(\xi_i)=  \nvol_{(X_i,\Delta_i)}(\xi_i) \leq \nvol_{(X_i,\Delta_i)}(\xi_0)=a_i(\xi_0)^n\vol(\xi_0) \leq a^n_0\vol(\xi_0)\]
  Hence, $a_i(\xi_i)^n\vol(\xi_i)$ is bounded by the right-hand-side. On the other hand, $\vol(\xi_i)$ diverges to $+\infty$. Hence, $a_i(\xi_i)$ must converge to $0$. Moreover, $a_i(\xi_0) \leq a_0$.  Therefore, all $a_i$ are contained in a compact truncation of $\hat \sigma^\vee$.  Then (possibly after passing to a subsequence) $a_i$ converges to an element $a_\infty \in \hat \sigma^\vee$ with $a_\infty(\xi_\infty)=0$.   Hence, if $\tau$ is the face of $\hat \sigma$ which contains $\xi_\infty$ in its relative interior, then $a_\infty \in \tau^*$.
\end{proof}

\begin{lemma}
  \label{lem:sequence-of-cones}
  Let $G_i$ be a sequence, whose members are finite sets of rays in $M_\RR$ and $\sigma \subset N_\RR$ a cone.  If  \(\dist(G_i, \sigma^\vee):=\max\{\dist(\rho, \sigma^\vee)\} \) converges to $0$, then for every element $w \in \rint \sigma$ there is an index $i_0$ such that for $i > i_0$ the element $w$ is contained in the  cone dual to the one spanned by the rays of $G_i$. In symbols
  \[\exists_{i_0} \forall_{i > i_o} \colon w \in \left(\sum_{\rho \in G_i}\rho \right)^\vee.\]
\end{lemma}
\begin{proof}
  After fixing an inner product we are going to identify $N_\RR$ and $M_\RR$ in the following.
  We set
  \[\epsilon=\dist( \sfrac{w}{|w|}, \partial \sigma)=\min\{\langle \sfrac{u}{|u|}, \sfrac{w}{|w|} \rangle \mid u \in  \sigma^\vee\}.\]
  Then starting from some index we have $\dist(\sfrac{u}{|u|}, \sigma^\vee) < \epsilon$ for any element $u$ of a ray in $G_i$. Now assume that $v \in \sigma^\vee$ with $|v|=1$ is the element of minimal distance to $\sfrac{u}{|u|}$. Then we have
  \begin{align*}
    \langle \sfrac{u}{|u|},   \sfrac{w}{|w|}\rangle &= \langle (\sfrac{u}{|u|}-v) + v,  \sfrac{w}{|w|} \rangle\\
                                                    &=  \langle \sfrac{u}{|u|}-v,  \sfrac{w}{|w|} \rangle + \langle v,  \sfrac{w}{|w|} \rangle\\
                                                    &=  \pm \dist( \sfrac{u}{|u|} - v ,w^\perp) + \langle v,  \sfrac{w}{|w|} \rangle\\
                                                    & > -\epsilon + \epsilon =0
  \end{align*}
  Here, the inequality in the last line follows from
  \[\dist( \sfrac{u}{|u|}- v ,w^\perp) \leq |\sfrac{u}{|u|}-v| < \epsilon.\]
  Hence, $w$ evaluates positively with every ray generator in $G_i$. Therefore, $w$ is contained in the cone dual the one spanned by the rays in $G_i$.
\end{proof}


\begin{lemma}
  \label{lem:sequence-of-boundaries2}
  In the situation described at the beginning of Section~\ref{sec:log-boundedness} let  $v$ be an element of a face $\tau \prec \hat \sigma$,
$w \in \operatorname{relint}(\hat \sigma)$ and $v_i \in [v,w]$ a sequence on the connecting line seqment with $v_i \to v$.
If the sequence $a_i$ corresponding to the log discrepancies with respect to $\Delta_i$ converges to an element $a_\infty$ in  the dual face $\tau^*$, then we have
\[\nvol_{(X_i,\Delta_i)}(v_i) \to 0.\] 
\end{lemma}
\begin{proof}
  Following \cite[3.2.2]{li2017stability} we may express $\nvol_{(X_i,\Delta_i)}$ via the theory of Newton-Okounkov bodies. Fix some element $\X_s$ of our family.  Assume $\dim(X_s)=n$ and $\dim \hat T = r$. Then by \cite[3.2.2]{li2017stability} there exists a convex cone
  $\Sigma \subset \hat M_\RR \times \RR^{n-r}$ such that $P(\Sigma)=\hat \sigma^\vee$ and
  \begin{equation}
  \vol_{X_i}(\xi)= \vol_{\X_s}(\xi)=\vol \{u \in \Sigma \mid \langle P(u), \xi \rangle \leq 1\}.\label{eq:vol-okounkov}
\end{equation}
  for every $\xi$ from the interior of $\hat \sigma$.  Here $P \colon \hat M_\RR \times \RR^{n-r} \to \hat M_\RR$ is the projection to the first factor.
  On the other hand, by  \cite[(16)]{li2017stability} the right-hand-side of (\ref{eq:vol-okounkov}) can be calculated as
  \begin{equation}
    \frac{1}{n}\int_{\tilde \Pi}\frac{\det(\tilde u, \partial_{y_1} \tilde u, \ldots, \partial_{y_{n-1}} \tilde u) }{\langle P(\tilde u(y)),\xi \rangle^n}dy,\label{eq:vol-integral}
  \end{equation}
  where $\Pi=\{ u \in \hat\sigma^\vee \mid \langle \xi',u \rangle=1\}$ for some fixed element $\xi'$ from the interior of $\hat\sigma$ and $\tilde u \colon \tilde \Pi \to
  P^{-1}(\Pi) \subset \Sigma^\vee$ with $\tilde \Pi \subset \RR^n$ being some parametrization.
  
  Now putting together (\ref{eq:vol-okounkov}) and (\ref{eq:vol-integral}) we obtain
  \[
    \lim_{i\to \infty}\nvol_{(X_i,\Delta_i)}(v_i)= \lim_{i\to \infty} \langle a_\infty , v_i \rangle^n \vol_{X_i}(v_i)= \lim_{i\to \infty} \int_{\tilde \Pi}\frac{\det(\tilde u, \ldots)}{n} \frac{\langle a_\infty , v_i \rangle^n}{\langle P(\tilde u(y)),v_i \rangle^n}dy
  \]
  For $P(\tilde u(y)) \notin \tau^*$ the expression under the integral converges pointwise to $0$. Now we write $v_i = \lambda_iw +(1-\lambda_i)v$ and obtain
  \[\frac{\langle a_\infty , v_i \rangle}{\langle P(\tilde u(y)),v_i\rangle}
    =\frac{\lambda_i \langle a_\infty , w \rangle}{\lambda_i \langle P(\tilde u) , w \rangle + (1-\lambda_i)\langle P(\tilde u) , v \rangle}
    = \frac{\langle a_\infty , w \rangle}{\langle P(\tilde u) , w \rangle + \frac{1-\lambda_i}{\lambda_i}\langle P(\tilde u) , v \rangle}.
\]
We see that the right-hand-side is bounded from above by
\[C:=\frac{\langle a_\infty , w \rangle}{\min \{\langle u,w \rangle \mid u \in \Pi\}}.\]
Note, that the denominator does not vanish, because $w$ was chosen from the interior of $\hat \sigma$ and $u\in \Pi \subset \hat\sigma^\vee$.

Now Lebesgue's Dominated Convergence Theorem tells us that the limit of integrals equals the integral of the pointwise limit. Since the limiting function vanishes outside a measure-zero set we get the desired result.
\end{proof}

\begin{proposition}
  \label{prop:log-bounded-deformation+k-semistable}
  Assume a class $\mathcal{P}$ of K-semistable log Fano cone $\TT^\ell$-singularities with normalized volume bounded away from $0$ admit $\TT$-invariant log bounded special degenerations, such that the bounding family of cone singularities arises from a $\QQ$-polarized family of projective varieties. Then $\mathcal{P}$ is also log bounded. 
\end{proposition}
\begin{proof}
  By assumption there exists a log bounded family, such that every element from $\mathcal{P}$ admits a $\TT$-invariant special degeneration to an element of the that family.  Now,
  for $r=\ell,\ldots, n$  let $\mathcal{Q}_r$ be the set of elements of that family which have a maximal torus of dimension $r$ in their automorphism group. Then by Proposition~\ref{prop:equivariantly-bounded} and Theorem~\ref{thm:zhuan-main}, there exists a $\TT^r$-invariant log bounding family $(\X, \mathcal{B})$ for $\mathcal{Q}_r$.

  Now, fix $r$, set $\hat\TT =\TT^r$ and consider the $\hat \TT$-invariant family  $(\X, \mathcal{B})$.
  For each of the finitely many subtori $\TT \subset \hat \TT$ with $\TT \cong \TT^\ell$ from Proposition~\ref{prop:finite-number-of-tori-pairs} all pairs $(X,B)$ of cone $\TT$-singularities having a $\TT$-invariant special degeneration to an element of $(\X, \mathcal{B})$  are parametrised by the multigraded Hilbert scheme $\Hilbp$, where $(X_0,B_0)$ is any member of $(\X, \mathcal{B})$.  Since $r$ takes finitely many values and for each $r$ we have finitely many subtori $\TT \subset \TT^r$ to consider we end up with a finite number of Hilbert schemes which form our log bounding family for  $\mathcal{P}$.
\end{proof}

\begin{corollary}
  \label{cor:semi-stable-bounded}
 Let $\epsilon, A > 0$, $n \in \mathbb{N}$ and let $I \subset [0, 1] \cap \QQ$ be a finite set. Then the class
 of K-semistable $n$-dimensional Fano cone $\TT$-singularities $x \in (X, \Delta)$ with $\nvol(x, X, \Delta) \geq \epsilon$, $\mldK^{\TT}_x(X, \Delta) \leq A$ and coefficients of $\Delta$ being from $I$ is log bounded.

 In particular, if there exists an $A$ such that $\mldK^{\TT}_x(X, \Delta) \leq A$ holds for every log Fano cone singularity of dimension $n$. Then the class of $n$-dimensional K-semistable log Fano cone singularities $(X,\Delta)$ with $\nvol(X,\Delta) > \epsilon$ and coefficients of $\Delta$ in $I$ is log bounded. 
\end{corollary}
\begin{proof}
  By Theorem~\ref{thm:zhuan-main} the class is log bounded up to $\TT$-invariant special degenerations. Now, the claim follows from Proposition~\ref{prop:log-bounded-deformation+k-semistable}.
\end{proof}

\section{Boundedness for K-semistable \texorpdfstring{$\mathbb{T}$-singularities of complexity one}{T-singularities of complexity one}}
\label{sec:t-varieties}
Our aim here is to bound the minimal log discrepancies of log Fano cone $\TT$-singularities of complexity $1$ and then apply Corollary~\ref{cor:semi-stable-bounded} in order to prove Theorem~\ref{thm:n-dim-K-ss-comp-1}.

\medskip
First,we recall some more details about $\mathbb{T}$-varieties
of complexity $1$. Our main reference will be the survey \cite{AIPSV12}. In this section we may specialize the construction introduced in Section~\ref{def:tvars} to the case of $Y=\PP^1$, i.e. the residual quotient of our $\TT$-variety is the projective line. In this case the condition for properness simplifies drastically. Indeed,  a polyhedral divisor $\mathcal{D}$ over $\PP^1$  is proper if the \emph{degree} $\deg \D:= \sum_{y\in \PP^1} \D_y$ is a proper subset of the tail cone, in symbols $\deg \D \subsetneq \tail \D$.

The following theorem is a particular case Theorem~\ref{thm:AH-main-result} in combination with \cite[Cororllary 5.8.]{LS13}.

 \begin{theorem}
 Let $X$ be a $\TT$-Fano cone singularity of complexity one. Then $X \cong X(\D)$ for some proper polyhedral divisor $\D$ on $(\pp^1,N_\qq)$ with complete locus. \end{theorem}

\medskip

In order to study discrepancies over klt singularities also need a global description for non-affine $\TT$-varieties. Here, we follow the notation from \cite[Sec.~5]{AIPSV12}. We fix a fan $\Sigma$ over the lattice $N$ and for every $y \in \PP^1$ a polyhedral complex $\CS_y$, such that $\Sigma=\tail \CS_y:=\{\tail \Delta \mid \Delta \in \CS_y
\}$ and for all, but finitely many we actually have $\CS_y=\Sigma$. Moreover, let $\mathfrak{d} \subset N_\RR$ be a subset, such that for every maximal cone $\sigma \in \Sigma$ with $\sigma \cap \mathfrak{d} \neq \emptyset$  and every $x \in \PP^1$ there is a unique polyhedron $\Delta^\sigma_y \in \CS_y$ with $\sigma=\tail(\Delta^\sigma_y)$ and 
\begin{equation}
\sigma \cap \mathfrak{d} = \sum_{y \in \PP^1} \Delta^\sigma_y  \subsetneq \sigma.\label{eq:f-divisor}
\end{equation}
Then we denote this data by $\CS=(\sum_y \CS_y \{y\},  \mathfrak{d})$ and call it an \emph{f-divisor} (or fansy divisor). The fan $\Sigma$ is called the \emph{tail fan} of $\CS$ and the set $\{y \in \PP^1 \mid \CS_y \neq \Sigma\} \subset \PP^1$ the \emph{support} of $\CS$. Moreover, the $\mathfrak{d}$ is called the degree of $\CS$, which we also denote by $\deg(\CS)$. The condition~(\ref{eq:f-divisor}) ensures that $\D^\sigma:=\sum_y \Delta_y^\sigma \{y\}$ defines a proper polyhedral divisor whenever $\sigma \cap \deg(\CS) \neq \emptyset$. On the other hand, when
$\sigma \cap \deg(\CS) = \emptyset$ and $\Delta \in \CS_y$ for some $y \in \supp(\CS)$ then
\[\D^{y,\Delta} \quad := \quad  \Delta \cdot \{y\} \quad + \sum_{z \in \supp(\CS) \setminus \{y\}} \emptyset \cdot \{z\}.\]
defines a proper polyhedral divisor with locus $\PP^1 \setminus \supp(\CS) \cup \{y\}$. The set of all p-divisors
$\D^\sigma$ and $\D^{y,\Delta}$ is called a \emph{divisorial fan} in the literature. By \cite[Prop.~20]{AIPSV12}
the affine varieties $X(\D^\sigma)$ and $Y(\D^{y,\Delta})$ alltogether from an open affine cover of a $\TT$-variety of complexity $1$, which we denote by $X(\CS)$. Note, in the literature the cones $\sigma$ of the tail fan which fulfill $\sigma \cap \mathfrak{d} \neq \emptyset$ are sometimes called \emph{marked}.

We can use f-divisors to study birational modifications of an affine $\TT$-variety $X(\D)$. By \cite[11.3]{AIPSV12} an f-divisor $\CS$ with $\D_y = |\CS_y|:=\bigcup_{\Delta \in \CS_y} \Delta$ defines a proper birational morphism $X(\CS) \to X(\D)$. In particular, we are able to calculate discrepancies of exceptional divisors on these modifications.

\begin{notation}
Given a polyhedron $P\subset N_\qq$, we denote by $P^{(k)}$ the set of $k$-dimensional faces of $P$.
Given an element $v\in N_\qq$, we denote by
$\mu(v)$ the smallest positive integer for which
$\mu(v)v\in N$.
\end{notation}

Let $\CS$ be an f-divisor on $\pp^1$ 
and let $\rho \in (\tail \CS)(1)$ such that $\deg(\CS) \cap \rho =\emptyset$. Then we have an associated $\mathbb{T}$-invariant prime divisor $D_\rho$ on $X(\D)$ (see, e.g.,~\cite[Proposition 2.23]{BGLM21}). The torus invariant prime divisors of this kind are called {\em horizontal divisors}.  

On the other hand, the choice of 
$z\in \pp^1$ and $v\in \D^{(0)}_z$ induces a so-called \emph{vertical} $\mathbb{T}$-invariant prime divisor $D_{z,v}$
(see, e.g,~\cite[Proposition 2.23]{BGLM21}). 

As a conseqence we can write ever torus invariant divisor on $X(D)$ as a sum
\[
  \sum_{\rho} c_\rho D_\rho + \sum_{y,v} c_{y,v} D_{y,v},
\]
where $\rho$ is running over all rays of $\tail(\D)$ with $\rho \cap \tail(\D) = \emptyset$ and
$(y,v)$ over all pairs with $y \in \PP^1$ and $v \in \D_y^{(0)}$.

$\TT$-invariant principal divisors correspond to semi-invariant functions. More precisely by \cite[Prop-~3.14]{PS11}  we have
\begin{equation}
  \label{eq:prinicpal-div}
  \operatorname{div}(f\cdot \chi^u) = \sum_{\rho} \langle n_\rho, u \rangle D_\rho + \sum_{y,v}  \mu(v) \left( \ord_y(f) + \langle v , u \rangle \right) \cdot D_{y,v}.
\end{equation}
Moreover, by  \cite{PS11} a canonical divisor on $X=X(\CS)$ is given by
\begin{equation}
  \label{eq:canonical-div}
  K_X =-\sum_{\rho}  D_\rho + \sum_{y,v} \mu(v)\left(k_y + \frac{\mu(v) - 1}{\mu(v)}\right) \cdot D_{y,v},
\end{equation}
where $\sum_y k_y \{y\}$ is any canonical divisor on $\PP^1$ (i.e. any divisor of degree $2$).
Note, that the coefficients of a principal divisor at $D_{y,v}$ are given by affine linear function $\langle \cdot , u \rangle +\ord_y(f)$. where the translation depends on $y$, but the linear part $\langle \cdot , u \rangle$ does not. This leads to the following characterisation of locally principal, i.e. Cartier, divisors. The proposition below summarizes results which appeared in slightly different formulations in  \cite[Prop~3.10, Cor.~3.28]{PS11}, \cite[Sec~2]{proj-tvar} and \cite[7.4, 8.2]{AIPSV12}.
\begin{proposition}
  \label{prop:cartier+ample}
  A Weil divisor 
  \[
    D=\sum_{\rho} c_\rho D_\rho + \sum_{y,v} c_{y,v} D_{y,v},
  \]
  on a complexity-$1$ $\TT$-variety $X(\dfan)$ is $\QQ$-Cartier if and only if there exists a piecewise linear function $h$ on $\tail \dfan$ and piecewise affine linear functions $h_y$ on each $\dfan_y$ with
  \begin{enumerate}
  \item The linear part of $h_y$ coincides with $h$. By this we mean that for every $\Delta \in \dfan_y$ the affine linear restriction $h_y|_{\Delta} = \langle \cdot , u \rangle +a_y$ fulfills $h|_{\tail \Delta}=\langle \cdot , u \rangle$.
  \item If $\sigma \in \tail \dfan$ is a maximal cone with $\sigma \cap \deg(\dfan)\neq \emptyset$. Then
    $\sum_y a_y^\sigma=0$, where $a_y^\sigma$ is given by the assumption $h|_{\Delta^\sigma_y} = \langle \cdot , u \rangle +a_y^\sigma$ with $\Delta^\sigma_y$ being the (unique) polyhedron with tail cone $\sigma$ in $\dfan_y$.\label{item:marked-principal}
  \item $c_\rho=-h(n_\rho)$ and $c_{y,v}=-\mu(v)h_y(v)$ for all invariant prime divisors $D_\rho$ and $D_{y,v}$.
  \end{enumerate}

  Moreover, the divisor $D$ is ample if and only if all functions $h_y$ from above are strictly concave and
for all maximal cones $\sigma \in \tail \dfan$ we have $\sum_y a_y^\sigma < 0$, where 
the $a^\sigma_y$ are as in Item~(\ref{item:marked-principal}) above.
\end{proposition}

\begin{definition}
Let $\mathcal{D}$ be a p-divisor
on $\pp^1$ with complete locus
and let $y\in \pp^1$.
We consider the positive integer:
\[
\mu_\mathcal{D}(y):=\max \{ \mu(v)\mid v\in \mathcal{D}_y^{(0)}\}. 
\] 
We call this integer the {\em multiplicity of $\mathcal{D}$ over $y$}.
Then, we can consider the following effective divisor on $\pp^1$:
\[
B:=\sum_y b_y \cdot y = 
\sum_y \frac{\mu_\mathcal{D}(y)-1}{\mu_\mathcal{D}(y)} \cdot \{y\}. 
\]
The pair $(\pp^1,B)$ is the quotient log pair
induced by the p-divisor $\mathcal{D}$.

For a boundary divisor $\Delta= \sum_{\rho} c_\rho D_\rho + \sum_{y,v} c_{y,v} D_{y,v}$ on $X=X(\D)$ we set
\[B_\Delta :=\sum_y b_{\Delta,y} \cdot y  :=
  \sum \max \left\{\frac{\mu(v)-1+c_{y,v}}{\mu(v)}  \;\middle|\; v \in \D_y^{(0)}\right\}\cdot y.
\]
The pair $(\PP^1,B_\Delta)$ is the quotient log pair pair induced by $(X,\Delta)$.
\end{definition}

\begin{proposition}
  \label{prop:discrepancies}
  Let $\D$ be a proper polyhedral divisor on $\PP^1$ with complete locus and $\Delta= \sum_{\rho} c_\rho D_\rho + \sum_{y,v} c_{y,v} D_{y,v}$ a boundary divisor on $X=X(\D)$.  Assume that $(X,\Delta)$ is $\QQ$-Gorenstein and  $\widetilde{X}=X(\CS) \to X(\D)$ is a birational modification corresponding to an f-divisor $\CS$. For every canonical divisor $K=\sum_y k_y \cdot \{y\}$ on $\PP^1$ there exists $u \in M_\QQ$ and a $\QQ$-divisors $\sum_{y \in \mathbb{P}^1} a_y \cdot \{y\}$ of degree zero on $\pp^1$, such that 
\begin{equation}
    \langle u, v_y \rangle + a_y = k_y + \frac{\mu(v_y)-1+c_{y,v}}{\mu(v_y)}= k_y + b_{\Delta,y},\label{eq:vertical-q-gorenstein}
  \end{equation}
    holds for every vertex $v_y$ of any polyhedral coefficient $\D_y$ and
    \begin{equation}
      \langle u, n_\rho \rangle  = -1 + c_\rho\label{eq:horizontal-q-gorenstein}
    \end{equation}
    holds for every primitive generator $n_\rho$ of a ray $\rho$ of $\tail \D$ with $\rho \cap \deg \D = \emptyset$.
    Moreover, for every vertical exceptional divisor $D_{z,w}$ and for every horizontal exceptional divisor $D_\tau$ the log discrepancies are given by
    \begin{equation}
      \label{eq:discr-vertical}
      \aXD(D_{z,w})=\mu(w)\langle -u, w \rangle + \mu(w)a_z
    \end{equation}
    and
    \begin{equation}
      \label{eq:discr-horizontal}
      \aXD(D_\tau)=\langle -u, n_\tau \rangle,
    \end{equation}
    respectively.
  \end{proposition}
  \begin{proof}
    By \cite[Prop. 3.1]{PS11} every Cartier divisor on $X=X(\D)$ is in fact principal. Hence, it has the form as in (\ref{eq:prinicpal-div}). Assume now $\operatorname{div}(f\cdot \chi^u)=m(K_X+\Delta)$ with $K_X$ as in (\ref{eq:canonical-div}).
    If we set $\sum_y a_y \cdot \{y\} = \sfrac{1}{m}\cdot \operatorname{div}(f)$, then comparing the coefficients of $\sfrac{1}{m}\operatorname{div}(f\cdot \chi^u)$ in front of $D_\rho$ and $D_{y,v}$, respectively, in (\ref{eq:prinicpal-div}) with those of $K+\Delta$ leads to the equations (\ref{eq:vertical-q-gorenstein}) and (\ref{eq:horizontal-q-gorenstein}), respectively.

    Now, comparing the coefficient of $D_{z,w} \subset X(\CS)$ and $D_\tau\subset X(\CS)$, resepctively, in equation (\ref{eq:prinicpal-div}) for $\operatorname{div}(f \cdot \chi^u)$ with that of $-K_{\widetilde{X}}$ as in (\ref{eq:canonical-div}) leads to the log discrepancies (\ref{eq:discr-vertical}) and (\ref{eq:discr-horizontal}), respectively.
  \end{proof}
\medskip
The following proposition allows us to 
check whether a log pair of complexity one
induced by a p-divisor with complete locus
is 
purely log terminal.

  \begin{proposition}
    \label{prop:log-terminal-pair}
    Let $\D$ be a p-divisor on $\pp^1$ with complete locus
    and let $X=X(\D)$ be the associated complexity one $\mathbb{T}$-variety and $\Delta$ a torus invariant boundary divisor.
    Then, the following statements hold.
   \begin{enumerate}
   \item The pair $(X,\Delta)$ is klt if and only if $(\PP^1,B_\Delta)$ is a klt log Fano pair;
     \label{item:klt-pair}
   \item the pair $(X,\Delta)$ is plt if and only if $(\PP^1,B_{\Delta})$ is a plt log Fano pair.
     \label{item:plt-pair}
 \end{enumerate}
\end{proposition}

\begin{proof}
  This  follows from Proposition~\ref{prop:discrepancies} as in the proof of~\cite[Theorem~4.1]{LS13} 
  replacing $B$ by $B_\Delta$ in the argument of the proof.
\end{proof}

We proceed to introduce the \emph{type} of a p-divisor on 
$\pp^1$.
The classification by types will allow us to prove
statements in a case-by-case analysis.

\begin{definition} 
Let $\mathcal{D}$ be a p-divisor on $\pp^1$.
We say that $\mathcal{D}$ is of {\em type} $(p,q,r)\in\nn^3$ if there exists 
three points $\{y_1,y_2,y_3\}\subset \pp^1$ satisfying the following:
\begin{itemize}
    \item the p-divisor $\mathcal{D}$ has multiplicity one at all points of 
    $\pp^1 \setminus \{y_1,y_2,y_3\}$ and 
    \item the set of multiplicities 
    $(\mu_\D(y_1),\mu_\D(y_2),\mu_\D(y_3))$ equals 
    $(p,q,r)$.
\end{itemize}
\end{definition} 

The following corollary follows from Proposition~\ref{prop:log-terminal-pair}.
\begin{corollary}
  Let $\mathcal{D}$ be a p-divisor on $\pp^1$. Assume that $(X(\mathcal{D}),\Delta)$ is klt for some torus invariant boundary divisor $\Delta$. 
  Then, the type of $\mathcal{D}$ is one of the platonic triples:
      \begin{equation}\tag{*}
      \label{eq:platonic-triple}
        (1,q,r),\; (2,2,r),\; (2,3,3),\; (2,3,4),\; (2,3,5).
    \end{equation}
\end{corollary}
    
We can extend the type of a p-divisor
to a set of vertices $(v_y)_{y\in \pp^1}$, with $v_y\in \mathcal{D}_y^{(0)}$, as follows.

\begin{definition}
Let $\mathcal{D}$ be a p-divisor on $\pp^1$.
A {\em family of vertices} of the p-divisor is a set 
$(v_y)_{y\in \pp^1}$ with 
$v_y\in \mathcal{D}_y^{(0)}$
for each $y\in \pp^1$.
Note that $v_y=0\in N_\qq$ holds for all but finitely many $y\in \pp^1$.
We say that the set $(v_y)_{y\in \pp^1}$
is of {\em type} $(p,q,r)\in \nn^3$ if the following conditions hold:
\begin{itemize}
    \item there are three points $\{y_1,y_2,y_3\}\subset \pp^1$ so that
    $\mu(v_y)=1$ for every $y\in \pp^1 \setminus \{y_1,y_2,y_3\}$, and 
    \item we have that 
    \[
    (\mu(v_{y_1}),\mu(v_{y_2}),\mu(v_{y_3}))=(p,q,r).
    \]
\end{itemize}
Note that, if $X(\mathcal{D})$ is log terminal, 
then $(p,q,r)$ must be one of the triples in~\eqref{eq:platonic-triple}.
\end{definition}

The following proposition allows us to 
check whether a log pair of complexity one is 
purely log terminal.

 \begin{corollary}
    \label{cor:log-terminal-pair}
    Let $X$ be an affine $\mathbb{T}$-variety
    of complexity one and $\Delta$ a torus invariant boundary divisor.
    Assume that $(X,\Delta)$ is klt.
    Let $D_\rho$ be a horizontal prime divisor on $X$
    and let $D_{z,v}$ be a vertical prime divisor on $X$, such that $\D_\rho,D_{z,v} \not \subset \supp \Delta$.
    Then, the following statements hold.
    \begin{enumerate}
    \item \label{cor:case1} $(X,\Delta+ D_\rho)$ is plt.
    \item \label{cor:case2} the pair  $(X,\Delta + D_{z,v})$ is plt if and only if $\sum_{y \neq z} b_{\Delta,y}< 1$.
    \end{enumerate}   
  \end{corollary}
  
  \begin{proof}
    If $X=X(\D)$ is given by a p-divisor $\mathcal{D}$ with affine locus, then $(X,D)$ is locally toric by~\cite[Example 2.5]{LS13} for $D=D_\rho$ or $D=D_{z,v}$. Then, both claims follow from the fact that a toric pair $(X,D)$ is plt if $D$ is an irreducible and reduced toric divisor. Hence, we may assume that  $\D$ is a p-divisor with complete locus $\PP^1$. The first claim follows from
    Proposition~\ref{prop:log-terminal-pair}, since $B_{\Delta}=B_{\Delta + D_\rho}$ by definition. For the second claim note that $(\PP^1, B_{\Delta + D_{z,v}})$ is plt if any only if
    \[\deg B_{\Delta + D_{z,v}} =  b_{\Delta+D_{z,v},z} + \sum_{y \neq z} b_{\Delta,y} = 1 + \sum_{y \neq z} b_{\Delta,y} < 2.\]

  \end{proof}

\subsection{\texorpdfstring{Koll\'ar components on $\mathbb{T}$-varieties}{Kollar components on T-varieties}}
\label{sec:kollar-components-T-varieties}

In this section, we prove statements
regarding the existence of torus invariant Koll\'ar components on $\mathbb{T}$-singularities of complexity one. For this we need to decribe an equivariant  proper birational morphism $X' \to X=X(\D)$. We will describe $X'$ by an f-divisor $\mathcal S = \sum_y \mathcal S_y$ as discussed at the beginning of Section~\ref{sec:t-varieties}. The extracted prime divisor is either vertical, i.e. coressponding to a vertex $v \in \mathcal S_z^{(0)}$, or horizontal, i.e. corresponding to a ray $\rho \in \tail \mathcal S$.

We will first describe two constructions for such an $X'$, which extracts a vertical or a horizontal exceptional divisor respectively.

\begin{construction}[The f-divisor $\CS(z,v)$]
  \label{con:vertical-extraction}
  Consider a $\TT$-cone singularity $X=X(\D)$ of complexity $1$. Fix an interior point $v \in \D_z$.
  We consider the polyhedron  $\mathfrak d=v+(\sum_{y\neq z}\D_y)$. For a proper face  $F \prec \mathfrak d$ the cone over this face is given as $\sigma_F = \overline {\pos{F}}$. The \emph{face fan} $\Sigma$ of $\mathfrak d$ consist of all cones of the form  $\sigma_F$ and their faces.    Since $v \in \rint \D_z$, we have $\mathfrak d \subsetneq \deg \D \subsetneq \tail \D$ and $\Sigma$ is a polyhedral subdivision of $\tail \D$. 

     Every face $F \prec \mathfrak d$ will be a sum $F=\sum_{y\neq z} F_y$, where $F_y \prec \D_y$. We will say that $F_y$ is a Minkowski summand of $F$. For every $y\neq z$ we consider the polyhedral subdivisions $\dfan_y$ consisting of polyhedra of the form $\Delta_y^F :=F_y + \sigma_F$, where $F_y \prec \D_y$, 
     $F \prec \mathfrak d$ and
     $F_y$ is a Minkowski summand of $F$. For $y=z$, we consider the subdivision $\dfan_z$ of $\D_z$ consisting of the following polyhedra:
     \begin{itemize}
     \item a translated cone $\Delta_z^\sigma:=v + \sigma$ for every $\sigma \in \Sigma$, and
     \item $\overline{\conv(P \cup \{v\})}$ for every proper face $P \prec \D_z$.
     \end{itemize}
     Then, the f-divisor $\CS(z,v)=\left(\sum_y \CS_y, \mathfrak d\right)$ gives rise to proper birational morphism
     $X(\CS) \to X(\D)$ and  the only exceptional torus invariant prime divisor is $D_{z,v}$, since $v \in \CS_z^{(0)}$
     is the only new vertex of the $\CS_y$ (which are subdivisions of the $\D_y$).
   \end{construction}

   \begin{construction}[The f-divisor $\CS(\rho)$]
  \label{con:horizontal-extraction}
    Given a $\TT$-cone singularity $X=X(\D)$ of complexity $1$. We consider the fan $\Sigma$ obtained by star subdividing $\tail \D$ with respect to $\rho$. That is the fan $\Sigma = \{\tau + \rho \mid \tau \prec \tail \D\}$. Similar for every $y \in \pp^1$ we obtain a polyhedral subdivision $\dfan_y = \{F + \rho \mid F \prec \D_y\}$. Note, that for every $\tau \prec \tail \D$ with $\tau \cap \deg \D \neq \emptyset$ there is a unique face $F^\tau_y$ having tail $\tau$ in every $\D_y$. We set
    \[
      \mathfrak{d} = \bigcup_{\substack{\tau \prec \tail \D \\ \tau \cap \deg \D \neq \emptyset}} \sum_y(F_y^\tau + \rho). 
    \]
    Then the f-divisor $\CS(\rho)=(\sum_y \CS_y, \mathfrak d)$ gives rise to proper birational morphism $X(\CS(\rho)) \to X(\D)$ and  the only exceptional torus invariant prime divisor is $D_\rho$. Indeed, we have $\Sigma(1)=(\tail \D)(1) \cup \{\rho\}$ and $\dfan_y^{(0)}=\D_y^{(0)}$ for every $y\in \PP^1$.   
   \end{construction}
   
  \begin{proposition}
    \label{prop:akollar-vertical}
    Let $X=X(\D)$ be a log terminal singularity of complexity one.
    Then every rational point $v\in \rint \D_z$
    corresponds to a weak component
    $D_{z,v}$ over $X$.
    Furthermore, $D_{z,v}$ is a Koll\'ar component if and only if  $\sum_{y \neq z} b_{\Delta,y}< 1$.
    In this case, $X$ is of type $(1,p,q)$.
   \end{proposition}
   
   \begin{proof} We consider the f-divisor $\CS(z,v)=\left(\sum_y \CS_y, \mathfrak d\right)$ from Construction~\ref{con:vertical-extraction}. 
     Given a facet $F \prec \mathfrak d$ there is a unique linear map $h_F\colon N_\QQ \to \QQ$ with $h_F|_{F} \equiv 1$.  This will define a continuous piecewise linear 
     function $h$ on $\Sigma$ being constant along the boundary of $\mathfrak d$. By construction this is a strictly concave function. 
     
     Consider $y\neq z$, then for each $F_y$ being a Minkowski summand of some facet $F \prec \mathfrak d$, we have that $h_F|_{F_y}$ is constant. Indeed, it is constant on $F$. Let $a^F_y$ be given by $h_F|_{F_y}\equiv -a^F_y$. Then $h_{F,y}:=h_F +  a^F_y$ defines an affine linear function on $F_y + \sigma_F$ with $h_{F,y}|_{F_y} \equiv 0$. This defines a piecewise linear function $h_y$ on $\dfan_y$ which is constant $0$ along the boundary of $\D_y$.

     Now, consider $y= z$. For every cone $\sigma_F \in \Sigma$ as above we have a unique affine linear function $h_{F,z}=h_F + a^F_z$ with $h_{F,z}(v)=1$. Now, for every facet $P \prec \D_z$ there is a unique affine linear function $h_{P,v}$ with $h_{P,v}|_{P}\equiv 0$ and $h_{P,v}(v)=1$. Again this defines a concave piecewise affine function $h_z$ on $\mathcal{S}_y$ which is constant $0$ along the boundary of $\D_z$ and equal to $1$ at the interior point $v$. 
     By construction, 
     we have that
     $h_z(v)=1$
     and 
     $h_y(w)=0$ 
     for every $w\in \mathcal{D}_y^{(0)}$ with $y\neq z$ or $w \neq v$.
     Furthermore, we have that
     $h= 0$ along the boundary of $\tail \D$.
     Let us fix a maximal cone $\sigma_F \in \Sigma$ and a point $v_y \in F_y$ for $y \neq z$. Moreover, we are setting $v_z=v$ for $y=z$. Then we have
     $\sum_{y} v_y \in F$.   Since by construction $h_y|_{F_y} \equiv 0$, $h_z(v)=1$ and $h|_{F}\equiv 1$, we obtain
     \begin{align*}
       1+0 = h_z(v) + \sum_{y \neq z} h_y(v_y) &=\sum_y h_y(v_y)\\
                                              &=\sum_y \big(h(v_y) + a^F_y\big)\\
                                              &= h\Big(\sum_y v_y\Big) + \sum_y a^F_y\\
                                                &=1+\sum_y a^F_y .
     \end{align*}    Hence, we conclude that $\sum_y a^F_y  = 0$.   By Proposition~\ref{prop:cartier+ample} the divisor $-\mu(v)D_{z,v}$ is an ample $\qq$-Cartier divisor on $X(\dfan)$.
    
    The second statement of the proposition follows from Corollary~\ref{cor:log-terminal-pair} item~\eqref{cor:case2}.
    For the last statement recall that $b_{\Delta,y}=\max \left\{\frac{\mu(v)-1+c_y}{\mu(v)} \mid v \in \D_y^{(0)}\right\}$.
    Hence, $b_{\Delta,y} > \sfrac{1}{2}$ whenever $\mu(y) > 1$. Therefore, in order to have
    $\sum_{y \neq z} b_{\Delta,y}< 1$ there can be at most one $y \neq z$ with $\mu(y) > 1$. In other words we have $\mu(y')=1$ for $y' \notin \{z,y\}$. Hence, $X$ is of type $(1,p,q)=(1,\mu(z),\mu(y))$.
   \end{proof}
   \begin{proposition}
    \label{prop:kollar-horizontall}
    Let $X=X(\D)$ be a log terminal singularity of complexity one. Then  every primitive lattice point  $n_\rho \in \rint (\tail \D)$ with associated ray $\rho=\RR_{\geq 0}\cdot n_\rho$ corresponds to a Koll\'ar component $D_\rho$ over $X$.
  \end{proposition}

  \begin{proof} We consider the f-divisor from Construction~\ref{con:horizontal-extraction}. For every facet $\tau \prec \tail \D$  there is a unique linear map $h_\tau:N_\QQ \to \QQ$ with $h|_\tau \equiv 0$ and $h(\lambda n_\rho)=\lambda$. This gives rise to a strictly concave piecewise linear function $h$ on $\Sigma$.  Moreover,  for every facet $F\prec \D_y$ with tail cone $\tau$  there is a unique affine linear function $h_F \colon (F +\rho) \to \RR$ with $h_F|_F \equiv 0$. Furthermore, for every $v + \lambda n_\rho \in F + \rho$, we have $h_F(v+\lambda n_\rho)=h_F(v) +\lambda$.  Hence, we obtain a strictly concave piecewise affine linear function  $h_y$ on $S_y$. 
    Consider a maximal cone $\sigma = (\tau + \rho) \in \Sigma$. Then there is a unique polyhedron $\Delta_y=F_y+\rho$ in $\dfan_y$ with $\tail F_y=\tau$ and therefore $\tail \Delta_y=\sigma$.
    For every $y\in \pp^1$, we may pick a point $v_y \in F_y$. Then $h_y(v_y) = h(v_y) +a^\tau_y = 0$ and 
    \[
\sum_y h_y(v_y) = h\left(\sum_y v_y\right) + \sum_y a^\tau_y.
    \]
    Now, if $\tau \cap \deg \D \neq \emptyset$, then $\sum_y v_y \in \tau$. Since $h|_\tau\equiv 0$, we conclude that $\sum_y a^\tau_y =0$ and therefore $\sum_y a^\tau_y \cdot y$ defines a principal divisor on $Y$. Similarly, we see that in the case that $\tau \cap \deg \D =\emptyset$ we must have   $0 > \sum_y a^\tau_y$, since now $\sum_y v_y$ lies in the interior of $\sigma=\tau+\rho$, where $h$ is strictly positive.

    By construction, $h_y(v)=0$ for every vertex $v \in \D_y^{(0)}$, $h(n_{\rho})=1$ and $h(n_{\rho'})=0$ for every $\rho' \in (\tail \D) (1)$. By Proposition~\ref{prop:cartier+ample} the divisor $-D_\rho$ is $\QQ$-Cartier and ample.

    Hence, $D_\rho$ is a weak component. We may invoke Corollary~\ref{cor:log-terminal-pair} item~\eqref{cor:case1} to see that it is actually a Koll\'ar component.
  \end{proof}
  
  

  \subsection{Discrepancies of vertical Koll\'ar components}
  \label{subsec:bounding-mld-deg}
  
  In this subsection, we give a sufficient criterion for a $d$-dimensional log Fano $\mathbb{T}$-singularity 
  $(X,\Delta)$ of complexity $1$ and type $(1,p,q)$  to  admit a Koll\'ar component of log discrepancy $2d-1$.
  
  Let $\mathcal{D}$ be the p-divisor defining a log terminal singularity of complexity one and type $(1,p,q)$.
  We can fix two points $z,z'\in \pp^1$ for which
  $\mu_\mathcal{D}(z)=p$, 
  $\mu_\mathcal{D}(z')=q$,
  and $\mu_\D(y)=1$ for every
  $y\in \pp^1\setminus \{z,z'\}$.
  Every vertical Koll\'ar component of the form $D_{z,w}$ (or $D_{z',w'}$) induces a degeneration to a toric singularity (see, e.g.,~\cite{IV12}). The log discrepancy at $D_{z,w}$ (resp. $D_{z',w'})$ can then be calculated on the toric special fibre of the degeneration. More precisely, we have the following statement.

  \begin{proposition}
    \label{prop:kollar-degeneration}
  The degeneration induced by $E=D_{z,w}$ is the toric variety $X(\sigma_z)$ associated to the cone
  \[\sigma_z := \pos{\left(\D_z \times \{1\} \;\cup\; (\tail \D) \times \{0\} \;\cup\; \left(\sum\nolimits_{y\neq z}\D_y \right) \times \{-1\}\right)} \subset \hat N_\RR:=N_\RR \times \RR.\]
  Moreover,  the corresponding Koll\'ar component $E_0$ in $X(\sigma_z)$ is given by the ray \[
  \pos{\left(w,1\right)} \in \hat N^+_\RR:=N_\RR \times \RR^+.
\]
 Let $\Delta=\sum_{\rho} c_\rho D_\rho + \sum_{y,v}c_{y,v} D_{y,v}$ be a torus invariant boundary divisor on $X$. Then we have $a_{(X,\Delta)}(E)=a_{(X(\sigma_z), \Delta_z)}(E_0)$.
  Here, the toric divisor $\Delta_z$ on the toric special fibre is defined by
\[\Delta_z = \sum_{\hat\rho \in \sigma_z(1)} c_{\hat \rho} D_{\hat \rho},\]
   where
   
   \begin{equation}
     c_{\hat \rho} =
     \begin{cases}
     c_\rho & \text{ if }\hat\rho = \rho \times \{0\}, \\
     c_{z,v_z}& \text{ if }\hat\rho = \RR_{\geq 0} \cdot (v_z,1) \text{ for some } v_z \in \D_z^{(0)}, \\
     \sum_{y\neq z} c_{y,v_y} & \text { if } \hat \rho = \RR_{\geq 0} \cdot (\sum_{y\neq z} v_y,1) \text{ with } \forall_{y \neq z} \colon v_y \in \D_y^{(0)}.
     \end{cases}\label{eq:boundary-ray-cases}
   \end{equation}
\end{proposition}

\begin{proof}
 The existence of this flat degeneration follows from~\cite[Theorem 2.8]{IV12}.

It is easy to see, that the rays of $\sigma_z$ are indeed of one of the  three types listed in (\ref{eq:boundary-ray-cases}).

Let $\sum_y k_y \cdot \{y\} := -\{z\}-\{y'\}$ be a canonical divisor on $\PP^1$ with  $y'\in \PP^1 \setminus \{z\}$. By Proposition~\ref{prop:discrepancies} there exists $u \in M_\QQ$ and a $\QQ$-divisors $\sum_{y \in \mathbb{P}^1} a_y \cdot \{y\}$ of degree zero on $\pp^1$, such that 
\begin{equation}
    \langle u, v_y \rangle + a_y = k_y + \frac{\mu(v_y)-1+c_{y,v}}{\mu(v_y)},\label{eq:vertical-q-gorenstein1}
  \end{equation}
    holds for every vertex $v_y$ of any polyhedral coefficient $\D_y$ and
    \begin{equation}
      \langle u, n_\rho \rangle  = -1 + c_\rho \label{eq:horizontal-q-gorenstein1}
    \end{equation}
    holds for every primitive generator $n_\rho$ of a ray $\rho$ of $\tail \D$ with $\rho \cap \deg \D = \emptyset$.
    We set $\hat u= (u,a_z)$. Then
    \begin{enumerate}
    \item for $\hat \rho =\rho \times \{0\}$ Equation~(\ref{eq:horizontal-q-gorenstein1}) is equivalent 
      to $\langle\hat u,  n_{\hat \rho} \rangle=-1+c_\rho$,
    \item for $\hat\rho = \RR_{\geq 0} \cdot (v_z,1)$ with $v_z \in \D_z{(0)}$ the Equation~(\ref{eq:vertical-q-gorenstein1}) is equivalent to $\langle\hat u, n_{\hat \rho}\rangle= \langle\hat u, \mu(v_z)(v_z,1) \rangle=-1+c_{z,v_z}$ 
    \item and finally for $\hat \rho = \RR_{\geq 0} \cdot (\sum_{y\neq z} v_y,1)$ applying (\ref{eq:vertical-q-gorenstein1}) to each $v_y$ leads to  $\langle\hat u, n_{\hat \rho}\rangle= \langle\hat u, \mu(\sum v_y)(\sum_y v_y,-1) \rangle=-1+ \sum_{y \neq z} c_{y,v_y}$.
    \end{enumerate}
    Therefore, we obtain $a_{(X(\sigma_z), \Delta_z)}(E_0) = \langle - \hat u , \mu(w)(w,1) \rangle$. On the other hand, by Proposition~\ref{prop:discrepancies}
    the log discrepancy $a_{(X,\Delta)}(E)$ is given by \(\mu(w)\langle -u, w \rangle + \mu(w)a_z
=  \langle - \hat u , \mu(w)(w,1) \rangle,
    \)
   as well.
\end{proof}

\begin{remark}
  \label{rem:kollar-degeneration}
  From now on, given a p-divisor $\D$ on $(\pp^1,N)$ of type $(1,p,q)$, we 
  denote by $\sigma_z$ the cone in $N_\rr\times \rr$ provided by Proposition~\ref{prop:kollar-degeneration},
  i.e., the cone defining the central fibre of the toric degeneration.
  
  The special fibre of the degeneration only depends on $z$ and not on the choice of $w \in \D_z$. Having fixed the $z \in \pp^1$, the choice of a Koll\'ar component $D_{z,w}$ is then equivalent to the choice of a ray of the form $\rho=\pos{(w,1)}$ in the interior of $\sigma_z$. In other words, this means picking an interior ray of $\sigma_y$ which intersects $\hat N^+_\RR$.
\end{remark}

\begin{lemma}
  \label{lem:rays-of-degeneration}
  Let $\mathcal{D}$ be a p-divisor on $\pp^1$ of type $(1,p,q)$.
  Let $z,z' \in \pp^1$ for which 
  $\mu_\D(z')=p$ and $\mu_\D(z)=q$.
  Assume that $p\leq q$.
  The primitive lattice generators $\hat n_\rho$ of the rays $\rho \in \sigma_z(1)$ have the form $\hat n_\rho=(n_\rho,m_\rho) \in N \times \ZZ$ with $m_\rho \geq -p$. Moreover, there is at least one primitive ray generator with $m_\rho = q$.
\end{lemma}
\begin{proof}
  Recall that the rays of $\sigma_z$ are exactly the rays of the following three types:
  \begin{enumerate}
  \item rays of the form $\pos{(v,1)}$, where $v$ is a vertex of $\D_z$,\label{item:upper-ray}
  \item rays of the form $\pos{(\sum_{y\neq z} v_y,-1)}$, where for each $y \in \mathbb{P}^1 \setminus \{z\}$ the summand $v_y$ is a vertex of $\D_y$, and\label{item:lower-ray}
  \item rays of the form $\rho \times \{0\}$, where $\rho$ is an extremal ray of $\tail \D$ with $\rho \cap \deg \D = \emptyset$.\label{item:central-ray}
  \end{enumerate}

  Now,  there is at least one vertex $v \in \D_z$ with $\mu(v)=\mu_\D(z)=q$. Hence, the corresponding primitive lattice generator of $\pos{(v,1)}$ will be $(qv,q)$.
  In Case~(\ref{item:lower-ray}), we note that all the $\mu(v_y)$ are equal to one with the potential exception of $\mu(v_{z'}) \leq \mu_\D(z') = p$. Hence, $\mu(\sum_{y\neq z} v_y)=\mu(v_{z'}) \leq p$. For the corresponding ray generator we obtain
  $(\mu(v_{z'})\sum_{y\neq z} v_y,-\mu(v_{z'}))$ with $-\mu(v_{z'}) \geq -p$. Finally, in Case~(\ref{item:central-ray}), we have $\rho \subset N_\RR \times \{0\} \subset N_\RR \times [-p,+\infty)$.
  This finishes the proof of the lemma.
\end{proof}

\begin{proposition}
  \label{prop:1-p-q-vertical}
   Every $d$-dimensional
   log Fano cone singularity $(X(\D),\Delta)$ of type $(1,p,q)=(1,\mu(z'),\mu(z))$  with $mu(z')\leq \mu(z)$ admits a Koll\'ar component of log discrepancy $2d-1$ if  $\sum_{y \neq z} b_{\Delta,y}< 1$.
 \end{proposition}
 
 \begin{proof}
   By Proposition~\ref{prop:kollar-degeneration} and Remark~\ref{rem:kollar-degeneration} it is sufficient to find a ray $\rho$ in the interior of $\hat N_\RR^+ \cap \sigma_z$ such that for $X_0=X(\sigma_z)$ the log discrepancy of the corresponding Koll\'ar component $D_\rho$ is given by $a_{(X_0,\Delta_z)}(D_\rho) = 2d-1$. 

   Now, we may pick $d$ rays $\rho_1, \ldots, \rho_d \in \sigma_z(1)$, such that the corresponding primitive lattice points
   $n_{\rho_1},\ldots, n_{\rho_d}$ are linearly independent. By Lemma~\ref{lem:rays-of-degeneration}, we may assume that $n_{\rho_1} \in N \times \{q\}$ and $n_{\rho_i} \in N \times [-q,-1]$ for all $i=1,\ldots,d$. Then
   $n_\rho:=dn_{\rho_1} + \sum_{i=2}^d n_{\rho_d}$ spans a ray $\rho$ in $\hat N_\RR^+ \cap \sigma_z$. Moreover, there exists a unique element $-u \in \hat M_\QQ$ which calculates the log discrepancy of every divisor $D_\tau$ over $X_0=X(\sigma_z)$ as $\langle -u, n_\tau\rangle$. Then for the log discrepancy of $D_\rho$, we obtain
   
   \begin{align*}
     a_{(X_0,\Delta_z)}(D_\rho) = \langle -u, n_\rho \rangle &= d\langle -u, n_{\rho_1} \rangle  + \sum_{i=2}^d \langle -u, n_{\rho_d} \rangle\\
     &=  d\cdot a_{(X_0,\Delta_z)}(D_{\rho_1})   + \sum_{i=2}^d a_{(X_0,\Delta_z)}(D_{\rho_i})\\
     &\leq  d + (d-1) =2d-1.
   \end{align*}
   Here, the last inequality follows from the fact, that the $D_{\rho_i}$ are prime divisors \emph{on} $X_0$ and, therefore, have discrepancy $1-c_{\rho}\leq 1$, where $c_\rho$ is the coefficient of $\Delta_z$ at $D_\rho$.
 \end{proof}
 
\subsection{Discrepancies of horizontal Koll\'ar components}\label{subsec:discrepancy-horizonta-kollar}

In this subsection, we study the minimal log discrepancies
of complexity one $\mathbb{T}$-singularities 
of type $(2,2,r)$,
$(2,3,3)$, $(2,3,4)$ and $(2,3,5)$.

\begin{lemma}
  \label{lem:sum-of-horizontal-discrepancies}
  Let $X=X(\D)$ be a $\TT$-variety of complexity $1$ and $\Delta$ a $\TT$-invariant boundary divisor on $X$.
  Assume that $D_{\rho_1}$ and $D_{\rho_2}$ are two horizontal divisors over $X$. If $\rho$ is the ray spanned by by a primitive lattice point $\lambda_1 n_{\rho_1}+ \lambda_2 n_{\rho_2}$, then $D_{\rho}$ has log discrepancy
    \[\aXD(D_{\rho}) = \lambda_1 \cdot \aXD(D_{\rho_1}) +\lambda_2 \cdot \aXD(D_{\rho_2}).\]
  \end{lemma}
  \begin{proof}
  Since $K_X$ is $\qq$-Cartier, we can write $mK_{(X,\Delta)}={\rm div}(\chi^uf)$ for some $u\in M$ and $f\in \kk(\pp^1)$.
  Then, by the proof of~\cite[Proposition 3.4]{BGLM21}, we can write 
  \begin{align*}
       \aXD(D_\rho) & = \langle u/m, \lambda_1 n_{\rho_1}+\lambda_2n_{\rho_2} \rangle = \\
       \lambda_1 \langle u/m , n_{\rho_1}\rangle + \lambda_2 \langle u/m, n_{\rho_2}\rangle & =
  \lambda_1 \cdot \aXD(D_{\rho_1})+\lambda_2 \cdot \aXD(D_{\rho_2}).
  \end{align*}
  We conclude that the log discrepancy $\aX$ is linear on the horizontal 
  $\mathbb{T}$-invariant prime divisors.
  \end{proof}
  
  The following statement allows us to control the log discrepancy of horizontal divisors over $X(\D)$.
  
  \begin{lemma}
    Consider a p-divisor $\D$ on $\pp^1$ and a $\TT$-ivariant boundary divisor
    \[\Delta =\sum_{\rho}c_\rho D_\rho + \sum_{y,v} c_{y,v}D_{y,v} \] on $X(\D)$.
    Let $(v_y)_{y\in \pp^1}$ be a family of vertices of $\D$.
    Let $n_\rho$ be the primitive lattice generator of the ray $\rho$ generated by
    $\sum_y v_y \in \tail\D$.
    Then, the following statements hold:
    \label{lem:horizontal-log-discrepancies}
    \begin{enumerate}
    \item
     $\aXD(D_\rho)=\lambda_\rho \left(2-\sum_y \frac{\mu(v_y)-1+c_{y,v_y}}{\mu(v_y)}\right)$, where $n_\rho = \lambda_\rho \cdot \sum_y v_y$.
      \label{item:horizontal-log-discrepancies}     
    \item If $(v_y)_y$ is of type $(2,p,q)$ with $p,q \geq 2$, then
      $\aXD(D_\rho) \leq 2$.\label{item:log-discrepancy-2pq}
    \item If $(v_y)_y$ is of type $(1,p,q)$ then  $\aXD(D_\rho) \leq \frac{p+q}{\gcd(p,q)}$.
      \label{item:log-discrepancy-1pq}
\end{enumerate}
  \end{lemma}
  \begin{proof}
     Consider a $\qq$-divisor $\sum_{y\in \pp^1} k_y \{y\}$ on $\pp^1$ of degree $-2$.
    By Proposition~\ref{prop:discrepancies}, it follows from the $\QQ$-Gorenstein property that 
    there exists $u \in M_\QQ$ and a $\QQ$-divisors $\sum_{y \in \mathbb{P}^1} a_y \{y\}$ of degree zero on $\pp^1$, such that
    \[
    \langle u, v_y \rangle + a_y = k_y + \frac{\mu(v_y)-1-c_{y,v_y}}{\mu(v_y)},
    \]
    for every $y \in \mathbb{P}^1$. 
    Summing over $y \in \mathbb{P}^1$ then gives
    \begin{equation}
      \left\langle u, \;\sum_y v_y \right\rangle =  \sum_y\frac{\mu(v_y)-1+ c_{y,v_y}}{\mu(v_y)} -2 = \sum_y\frac{\mu(v_y)-1}{\mu(v_y)} -2+\sum_y\frac{c_{y,v_y}}{\mu(v_y)} .\label{eq:qcartier}    
    \end{equation}
    Now, by Proposition~\ref{prop:discrepancies} the log discrepancy  at $D_\rho$ is given by $-\langle u, n_\rho \rangle$. Hence, claim~(\ref{item:horizontal-log-discrepancies}) follows from the first identity in \eqref{eq:qcartier}.
    
    In the case of type $(2,p,q)$, then $\sum_y\frac{\mu(v_y)-1}{\mu(v_y)} -2$ is $-\frac{1}{q}$ if $p=2$ and
    $-\frac{1}{\lcm(2,p,q)}$ if $p=3$. Now, the log discrepancy at $D_\rho$ is calculated as $-\langle u, n_\rho \rangle$. On the other hand, $\rho = \pos {(\sum_y v_y)}$ and $\sum_y v_y \in \frac{1}{\lcm(2,p,q)}N$.
  Hence, $-\langle u, n_\rho\rangle \leq \lcm(2,p,q) \cdot \left\langle u\,, \;\sum_y v_y \right\rangle$ and the statement~(\ref{item:log-discrepancy-2pq}) follows.

  Similarly, for the type $(1,p,q)$ the quantity $\sum_y\frac{\mu(v_y)-1}{\mu(v_y)} -2$ in (\ref{eq:qcartier}) becomes $-\frac{p+q}{pq}$ and  $\sum_y v_y \in \frac{1}{\lcm(p,q)}N$. Therefore,
  we have 
  \[-\langle u, n_\rho\rangle \leq \lcm(p,q)\left(\frac{p+q}{pq} - \sum_y\frac{c_{y,v_y}}{\mu(v_y)}\right) \leq \frac{p+q}{\gcd(p,q)}.
  \]
  This shows statement (\ref{item:log-discrepancy-1pq}).
  \end{proof}
  
  \begin{notation} 
  We consider a polyhedral cone $\sigma \subset N_\QQ$. Let $\facets(\sigma)$ be the set of all facets of the cone $\sigma$.  If $\tail \Delta=\sigma$, then every facet $\tau \in \facets(\sigma)$ gives rise to a unique facet $F_\tau \prec \Delta$ such that $\tail F_\tau=\tau$. Indeed, let $n_\tau$ be the primitive inner normal to $\tau$, then $F_\tau$ is the face of $\Delta$ where $\langle n_\tau, \cdot \rangle$ takes its minimum.  For an element $v \in \Delta$, we now consider the set of facets $\tau \in \facets(\sigma)$ such that $v \in F_\tau$. We set:
  \[
  \facets(\Delta,v):=\{\tau \prec \sigma \mid v \in F_\tau\}.
  \]
  The element $v \in \sigma$ lies in the interior of $\sigma$ if and only if $\facets(\sigma, v)=\emptyset$. For two elements $v,v' \in \sigma$, we have
  \begin{equation}
    \facets(\sigma,v+v') = \facets(\sigma,v) \cap \facets(\sigma,v').
    \label{eq:facets-of-sums}
  \end{equation}
  Let $\D$ be a p-divisor on $\pp^1$.
  Then, we define the set: 
  \[
  \facets(\D) := \{\tau \in \facets(\tail \D) \mid \tau \cap \deg \D \neq \emptyset\}.
  \]
  \end{notation}

The following lemma follows from the introduced notation.

  \begin{lemma}
\label{lem:vertex-families-and-facets}
    Let $\D$ be a p-divisor on $\pp^1$.
    Let $(v_y)_{y\in \mathbb{P}^1}$ be a family of vertices of $\D$.  Then, the following statements hold.
  \begin{enumerate} 
  \item We have an identity 
    \[\facets\left(\deg \D, \sum_{y \in \mathbb{P}^1} v_y\right) = \bigcap_{y\in \mathbb{P}^1} \facets(\D_y,v_y).\]
  \item We have that 
    \[\facets\left(\tail \D, \sum_{y \in \mathbb{P}^1} v_y \right) = \facets(\D) \cap \bigcap_{y\in \mathbb{P}^1} \facets(\D_y,v_y)\subset  \bigcap_{y\in \mathbb{P}^1} \facets(\D_y,v_y) \]
    \label{lem:vertex-families-and-facets-tail}
  \end{enumerate}
\end{lemma}
\begin{proof}
We set $v = \sum_{y \in \mathbb{P}^1} v_y$. We have that $\tau \in \facets\left(\deg \D, v\right)$ if and only if
$\langle n_\tau ,w \rangle \geq \langle n_\tau ,v \rangle$ for all $w \in \deg \D$, i.e., for all $w=\sum_y w_y$ with $w_y \in \D_y$.  Clearly if $v \in  \bigcap_{y\in \mathbb{P}^1} \facets(\D_y,v_y)$, meaning that $\langle n_\tau, w_y \rangle \geq \langle n_\tau, v_y \rangle$ holds for $y$ and all $w_y \in \D_y$, then also
$\langle n_\tau ,\sum_y w_y \rangle \geq \langle n_\tau ,\sum_y v_y \rangle$. For the other direction assume there is a $z \in \PP^1$ and a $w_z \in \D_z$ such that $\langle n_\tau, w_z \rangle < \langle n_\tau, v_z \rangle$ holds. Then for $w := w_z + \sum_{y \neq z} v_y \in \deg D$ we have $\langle n_\tau ,w \rangle < \langle n_\tau ,v \rangle$. This proves the first claim.

For the second claim observe that $\facets\left(\tail \D, \sum_{y \in \mathbb{P}^1} v_y \right)  \subset \facets\left(\deg \D, \sum_{y \in \mathbb{P}^1} v_y\right)$, since $\deg \D \subset \tail \D$. Moreover, since by definition of $n_\tau$ the minimum of $\langle n_\tau , \cdot \rangle$ on $\tail \D$ is attained on $\tau$, we see that $\min \langle n_\tau , \deg \D \rangle$ and $\min \langle n_\tau , \tail \D \rangle$ differ if and only if $\deg \D \cap \tau = \emptyset$. Hence, a facet $\tau \in \facets\left(\deg \D, \sum_{y \in \mathbb{P}^1} v_y\right)$, is also contained in $\facets\left(\tail \D, \sum_{y \in \mathbb{P}^1} v_y \right)$ if only if $\tau \in \facets(\D)$.
\end{proof}

We turn to prove an arithmetic lemma that 
will be used in the proof of our main theorem.

\begin{lemma}
  \label{lem:technical-lemma}
  Consider three points $w, w'_1, w'_2 \in N_\QQ$ with $(w+w'_1)$ and $(w+w'_2)$ being linearly independent.
  Let $v_1$ and $v_2$ be the primitive lattice point in $\pos{(w + w'_1)}$ and  $\pos{(w +w_1')}$, respectively.
  Then there exist $0<\lambda_1 < \ell$ and $0<\lambda_2 < \frac{\ell^2}{\mu(w)}$, such that
  $\lambda_1v_1 + \lambda_2v_2$ is a lattice point, where $\ell:=\lcm(\mu(w_1'), \mu(w_2'))$.
\end{lemma}
\begin{proof}
  We set $w_1=w+w_1'$ and $w_2=w+w_2'$. We first show the claim for $\mu(w_1')=\mu(w_2')=1$. 
  In this case, we have $\mu(w_1)=\mu(w_2)=\mu(w)=:q$. Furthermore,  $qw_1 = p_1v_1$  and
  $qw_2 = p_2v_2$ with $\gcd(p_1,q)=\gcd(p_2,q)=1$.  Note, that \[p_1v_1-p_2v_2=qw_1 - qw_2 = q(w_1'-w_2')\] is divisible by $q$, since we assumed that  $\mu(w_1')=\mu(w_2')=1$.

  After restricting to the subspace spanned by $w_1$ and $w_2$ and to the corresponding sublattice, we may assume that $N \cong \ZZ^2$. By the above, $v_1$ and $v_2$ span a sublattice of index $q$ in $N$. Hence, after a suitable choice of coordinates, we may assume that $v_1=(1,0)$ and $v_2=(m,q)$ for some $0 \leq m < q$. Then, we have that 
  \[
  (1,1)=\frac{q-m}{q}v_1+\frac{1}{q}v_2.
  \]
Giving our claim by choosing $\lambda_1=\frac{q-m}{q}$ and $\lambda_2=\frac{1}{q}$.

For the general case, we first replace $N$ by $\bar N := \frac{1}{\ell}N$. With respect to this lattice, we are in the previous case. If we consider the multiplicity $\bar\mu (w)$ of $w$ with respect to the refined lattice $\bar{N}$, then we have $\bar \mu(w) \geq \mu(w)/\ell$. Hence, we know that there are $0<\lambda_1 < 1$ and $0<\lambda_2 < \frac{\ell}{\mu(w)}$ such that $\lambda_1v_1 + \lambda_2v_2 \in \bar N$. Then, we also have
$(\ell\lambda_1)v_1 + (\ell\lambda_2)v_2 \in N$. This finishes the proof of the lemma.
\end{proof}

\begin{proposition}
  \label{prop:1-p-q-horizontal}
  Assume $X=X(\D)$ is of type $(1,p,q)=(1,\mu(z'),\mu(z))$ with $\mu(z') \leq \mu(z)$ and $\Delta=\sum_{\rho}c_\rho D_\rho + \sum_{y,v} c_{y,v}D_{y,v}$ is a boundary divisor on $X$, such that all coefficients are either $0$ or larger than $\epsilon >0$.  Assume furthermore, that  the assumption of Proposition~\ref{prop:1-p-q-vertical} are \emph{not} fulfilled, i.e. $\sum_{y\neq z} \geq 1$. Then $X$ admits a horizontal Koll\'ar component $D_\rho$ with  $\aXD(D_\rho)< \sfrac{2}{\epsilon}$.
\end{proposition}
\begin{proof}
  Let $v_y$ the vertex of $\D_y$ achieving the maximum in the definition \[b_{\Delta,y}=\max \left\{\frac{\mu(v)-1+c_{y,v}}{\mu(v)} \;\middle|\; v \in \D_y^{(0)}\right\}.\]
  Since $\sum_{y\neq z} b_{\Delta,y} \geq 1$ holds, we must have
  \begin{equation}
  2- \sum_y \frac{\mu(v_y)-1+c_{y,v_y}}{\mu(v_y)}
= 2 - \sum_{y\neq z} b_{\Delta,y} - b_{\Delta,z} \leq 1 - \frac{\mu(v_z)-1}{\mu(v_z)}=\frac{1}{\mu(v_z)}.\label{eq:discrepancy-bound}
\end{equation}
Set $v=\sum_y v_y$. Since $v_y$ is a lattice point for $y\notin \{z,z'\}$, we conclude that $\mu(v_{z})\mu(v_{z'})v$ is a lattice point. Therefore, the lattice generator of the ray $\rho = \RR_{\geq 0}v$ has the form $n_\rho=\lambda_\rho v$ with $\lambda_\rho \leq \mu(v_{z})\mu(v_{z'})$. Now, from Lemma~\ref{lem:horizontal-log-discrepancies} we conclude that
\begin{equation}
\aXD(D_\rho) \leq \mu(v_{z'}) \leq \mu(z')=p.\label{eq:discrepancy-bound-by-p}
\end{equation}
Now, we have
\[\frac{p-1}{p} \leq b_{\Delta,z'} \qquad \text{and} \qquad \frac{p-1}{p} \leq \frac{q-1}{q} \leq b_{\Delta,z}.\]
Since $\sum_{y\neq z} b_{\Delta,y} \geq 1$ holds by assumption, there is at least one $w \notin \{z,z'\}$ such that
$b_{\Delta,w} = c_{w,v_{w}} > \epsilon$. Now, from the klt property of $(X,\Delta)$ we obtain by Proposition~\ref{prop:log-terminal-pair} the inequality
\[
2 > \sum_y b_{\Delta,y} = b_{\Delta,z} + b_{\Delta,z'} + b_{\Delta,w} \geq \frac{p-1}{p} + \frac{p-1}{p} + \epsilon.
\]
This implies $p < \sfrac{2}{\epsilon}$ and now by (\ref{eq:discrepancy-bound-by-p}) the claim $\aXD(D_\rho) \leq \sfrac{2}{\epsilon}$ follows.
\end{proof}

\begin{theorem}\label{theorem:mld-bound}
  There is a constant $A(d,\epsilon)$, dependening only on the dimension $d$ and a real number $\epsilon > 0$, satisfying the following. For a $d$-dimensional complexity one log Fano cone singularity $(X,\Delta)$ with coefficients of $\Delta$ being larger than $\epsilon$  there is a Koll\'ar component $E$ with log discrepancy $\aXD(E)<A(d,\epsilon)$.
\end{theorem}

\begin{proof}
  For $(X,\Delta)$ being of type $(1,p,q)$ this follows from Propositions~\ref{prop:1-p-q-vertical} and \ref{prop:1-p-q-horizontal}. 

  Now, we consider the remaining cases at once. Let $X=X(\D)$ be a complexity one $\mathbb{T}$-singularity of one of the remaining cases, i.e., of type $(2,p,q)$ with either $p=2$ and $q \geq 2$ arbitrary or $p=3$ and $q\in \{3,4,5\}$. Moreover, we have  $z \in \mathbb{P}^1$ with $\mu_\D(z)=q$.

  Our aim is now to construct a set $M$ of rays in $\tail \D$ such that: 
  \begin{itemize}
      \item we have that $\bigcap_{\rho} \facets(\tail \D, n_\rho)=\emptyset$,
      \item the cardinality $\# M$ is bounded by some constant $A(d)$, and
      \item for the log discrepancies we have
      $\aX(D_\rho)\leq A(d)$.
  \end{itemize}
 We set $\sigma = \tail \D$ and fix $n=d-1$ rays $R_1:=\{\rho_1, \ldots, \rho_{n}\} \subset \sigma(1)$ spanning a full-dimensional subcone of $\sigma$. Then every positive linear combination of the primitive generators $n_{\rho_1},\ldots,n_{\rho_n}$ lies in the interior of $\sigma$. It follows that $\bigcap_{\rho \in R} \facets(\sigma, n_{\rho})=\emptyset$.

 First, we consider the set 
 \[
 M_1=\{\rho \mid \rho \in R \text{ and } \deg \D \cap \rho = \emptyset \}. 
 \] 
 Those correspond to horizontal prime divisors $D_\rho$ in $X$. In particular, we have $\aXD(D_\rho) \leq 1$.

 For every ray $\rho \in R_2=R_1 \setminus M_1$, we have a family of vertices $(v^\rho_y)_{y\in \pp^1}$ with $\sum_{y\in \pp^1} v^\rho_y \in \rho$. Assume now such a family is of type $(2,p',q')$ with $p',q' \geq 2$.
 By Lemma~\ref{lem:horizontal-log-discrepancies}, we have $\aXD(D_\rho) \leq 2$. Let $M_2 \subset R_2$ be the set of those rays. Let  $R_3 = R_2 \setminus M_2$ then be the set of the remaining rays.

 Now, for every ray $\rho \in R_3$ there is a family $(v_y^\rho)_y$ of type $(1,r,s)$ with $\sum_{y\in \pp^1} v_y^\rho \in \rho$. If $\mu(v_z^\rho)=1$, then $r,s \leq 3$.
 Indeed, for every $y \in \mathbb{P}^1 \setminus \{z\}$, we have that $\mu_\D(y)$ is at most $p$ with $p \leq 3$. Hence, by Lemma~\ref{lem:horizontal-log-discrepancies}, we must have $\aXD(D_\rho) \leq 5$.
We set $M_3$ to be the set consisting of all such rays from $R_3$ and set $R_4=R_3 \setminus M_3$.

We consider a family $(t_y)_{y \in \mathbb{P}^1}$ of vertices $t_y \in \D_y$ of type $(2,p,q)$, which must exists as $X$ is of type $(2,p,q)$. This leads to a \emph{distinguished ray} $\tau = \pos{(\sum_y t_y)}$. By Lemma~\ref{lem:horizontal-log-discrepancies}, we have $\aXD(D_\tau) \leq 1$. Now, the remaing rays $\rho \in M_2$ correspond to families  $(v_y^\rho)_{y\in \pp^1}$  of type $(1,r_1,r_2)$ such that  $\mu(v_z^\rho)=r_i > 1$, with $i\in \{1,2\}$.
Without loss of generality, we assume $\mu(v_z^\rho)=r_1$. Then, we consider the two associated families $({v_y^{\rho_1}})_{y\in \pp^1}$ and $({v^{\rho_2}_y})_{y\in \pp^1}$ with
\[v_y^{\rho_1} =
  \begin{cases}
    v^\rho_z & y = z\\
    t_y & y\neq z
  \end{cases}, \hspace{5em}
  v_y^{\rho_2}=
  \begin{cases}
    t^\rho_z & y = z\\
    v_y & y\neq z
  \end{cases}
\]
and the associated rays $\rho_1=\pos{\left(\sum_y v_y^{\rho_1} \right)}$ and  $\rho_2=\pos{\left(\sum_y v_y^{\rho_2} \right)}$, respectively. Now, by Lemma~\ref{lem:vertex-families-and-facets}~(\ref{lem:vertex-families-and-facets-tail}) we have
\begin{align*}
  \facets(\sigma,n_{\rho_1})  \;\cap\;  \facets(\sigma,n_{\rho_2}) &= \facets(\D) \cap \bigcap_y \facets(\D_y,v_y^{\rho_1})  \;\cap\; \bigcap_y \facets(\D_y,v_y^{\rho_2})\\
  &= \facets(\D) \cap 
\bigcap_y \facets(\D_y,v_y^{\rho})  \;\cap\; \bigcap_y \facets(\D_y,t_y)\\
  &= \facets(\sigma, n_\rho)  \;\cap\; \facets(\sigma, n_\tau) \\
  &\subset  \facets(\sigma, n_\rho).
\end{align*}
By Lemma~\ref{lem:horizontal-log-discrepancies}, we have $\aXD(D_{\rho_1}) \leq 1$ and
$\aXD(D_{\rho_2}) \leq \frac{q+r_2}{\gcd(q,r_2)} \leq q+3$. The latter cannot be bounded by a constant dependending only on the dimension. However, note that
\begin{align*}
\sum_y v^{\rho_2}_y &= t_z + \sum_{y\neq z} v^{\rho_2}_y \qquad \text{ and }\\
\sum_y t_y &= t_z + \sum_{y\neq z} t_y.
\end{align*}
Hence, by applying Lemma~\ref{lem:technical-lemma}, we see that there is a lattice point $n_{\rho_2'}$ of the form
$n_{\rho_2'}=\lambda_1 n_{\tau} + \lambda_2 n_{\rho_2}$
with $0 < \lambda_1 \leq \lcm(p,q) \leq 6$ and $0 < \lambda_2 \leq  \frac{6^2}{q}$. Let $\rho_2'=\pos{n_{\rho_2'}}$. By Lemma~\ref{lem:sum-of-horizontal-discrepancies}, we have that 
\[
\aXD(D_{\rho_2'}) \leq 6+2\cdot 6^2 = 78.
\]
We now set $M_4=\{ \rho_1 \mid \rho \in R_4 \} \cup \{ \rho_2' \mid \rho \in R_4 \}$ and
$M = \bigcup_{i=1}^4 M_i$. Note, that
\[\facets(\sigma, n_{\rho_2'})
  \;=\; \facets(\sigma, n_{\rho_2}) \cap \facets(\sigma, n_{\tau}) \;\subset\; \facets(\D, n_{\rho_2})\]
holds. Hence, we have
\[
  \bigcap_{\rho \in M} \facets(\sigma, n_\rho)
  = \bigcap_{i=1}^4 \bigcap_{\rho \in M_i} \facets(\sigma, n_\rho) 
  \subset \bigcap_{\rho \in R} \facets(\sigma, n_\rho)  = \emptyset.
\]
Thus, $\eta := \pos \left(\sum_{\rho \in M} n_\rho\right)$ is a ray in the interior of $\sigma$. Therefore,
by Proposition~\ref{prop:kollar-horizontall} the prime divisor $E:=D_\eta$ over $X$ corresponds to a Koll\'ar component. Finally, Lemma~\ref{lem:sum-of-horizontal-discrepancies} implies that \[\aXD(E) \leq \# M \cdot 78 \leq 2(d-1)\cdot 78.\]
Indeed, we have $\# M \leq 2\cdot \# R_1 = 2(d-1)$, since $M=(R \setminus R_4) \cup M_4$ and $\# M_4 = 2\cdot \#R_4$. Moreover, for every $\rho \in M$ we have $\aXD(D_\rho) \leq 78$.
\end{proof}

\begin{proof}[Proof of Theorem~\ref{introthm:upper-bound-mld}]
This follows from Theorem~\ref{theorem:mld-bound}.
\end{proof}

\begin{proof}[Proof of Theorem~\ref{introthm:norm-vol-bound}]
  This follows from  Theorem~\ref{thm:zhuan-main} and Theorem~\ref{theorem:mld-bound}.
\end{proof}

We give an example that shows that
complexity one $\mathbb{T}$-singularities
with bounded normalized volume do not form a bounded family.
First, we need a lemma that allows us to determine
whether a family of complexity one $\mathbb{T}$-singularities is bounded.

\begin{lemma}\label{ex:bound-or-not}
Let $\{(X_i;x_i)\}$ be a bounded family of
complexity one $\mathbb{T}$-singularities. 
Let $\mathcal{D}_i$ be a p-divisor on $(\pp^1,N)$ defining
$(X_i;x_i)$, 
Then, up to isomorphism of p-divisors,
there are only finitely many possible polyhedra
$\mathcal{D}_{i,p}\subset N_\qq$ for $p\in \pp^1$. In paarticular, there occur only finitely many types $(p,q,r)$
\end{lemma}

\begin{proof}
Since $\{(X_i;x_i)\}$ belongs to a bounded family, 
we can find an affine space
\[
\mathbb{A}_{x_1,\dots,x_n}\times \mathbb{A}_{t_1,\dots,t_k}, 
\]
ideals $I_t=\langle f_1(x,t),\dots,f_s(x,t)\rangle$
depending on $t\in Z\subset  \mathbb{A}_{t_1,\dots,t_k}$
such that $(X_i;x_i)\simeq (V(I_{t_i});0)$ for some $t_i$.
Note that $(V(I_{t_i});0)$ admits the action of an algebraic torus $\mathbb{T}_i$. 
By~\cite[Theorem 1]{HM89}, it follows that we can find $\mathbb{T}_{i,0} \leqslant {\rm Aut}(\mathbb{A}_{x_1,\dots,x_n})$
so that $\mathbb{T}_{i,0}$ fixes $V(I_{t_i})$
and the induced action on $(V(I_t);0)$ equals the action of $\mathbb{T}_i$.
Indeed, the torus $\mathbb{T}_{i,0}$ 
is faithfully represented in the first cotangent
space of $V(I_{t_i})$ at the origin.
Passing to an open subset $Z_U \subset Z \subset \mathbb{A}_{t_1,\dots,t_k}$, we may assume that each 
$I_{t_i}$ is defined by the same number of polynomials
$f_j$'s and the same monomials on $x$
appear on all the $f_j$'s.
Hence, we can find a torus $\mathbb{T}\leqslant {\rm Aut}(\mathbb{A}_{x_1,\dots,x_n})$ for which
$\mathbb{T}|_{V(I_{t_i})} = \mathbb{T}_i$ for each $i$.
Now, let $\mathcal{D}$ be the 
p-divisor on $(Y,N_0)$ defining 
the $\mathbb{T}$-variety $\mathbb{A}_{x_1,\dots,x_n}$.
The normalized Chow quotient
of $(V(I_{t_i});0)$ by $\mathbb{T}_i$
gives a rational curve $Y_i \simeq \pp^1 \subset Y$.
Furthermore, the p-divisor $\mathcal{D}_i$
is isomorphic to the pull-back of
$\mathcal{D}$ to $Y_i$.
Since $\mathcal{D}$ has only finitely many different polyhedra, it follows that each pull-back
has only finitely many different polyhedra.
This finishes the proof.
\end{proof}

\begin{example}\label{ex:not-bounded}
 Once again we consider the toric variety $X_0$ from Example~\ref{ex:running} and the deformations $\X_n$.
 Note that the generic fibre $X_{n,1}:=(\X_{n})_{t=1}$ is a complexity one $\mathbb{T}$-singularity.  The p-divisors $\D^n$ for these Fano cone singularities are given by
 \[
    \D^n=(\left(\begin{smallmatrix}
       1/2\\0
     \end{smallmatrix}\right) + \sigma)
   [0] +
    (\left(\begin{smallmatrix}
       -k/(2k+1)\\0
     \end{smallmatrix}\right) + \sigma)[1] + (\overline{
     \left(\begin{smallmatrix}
       0\\0
     \end{smallmatrix}\right)
        \left(\begin{smallmatrix}
       0\\1
     \end{smallmatrix}\right)}+\sigma)[\infty]
 \]
 with tailcone $\sigma$ being spanned by $e_1$ and $e_1+(4k+2)e_2$ if $n=2k+1$ and
  \[
    \D^n=
 (\overline{
     \left(\begin{smallmatrix}
       0\\0
     \end{smallmatrix}\right)
        \left(\begin{smallmatrix}
       0\\1
     \end{smallmatrix}\right)}+\sigma)[0]
 + (\overline{
     \left(\begin{smallmatrix}
       0\\0
     \end{smallmatrix}\right)
        \left(\begin{smallmatrix}
       0\\1
     \end{smallmatrix}\right)}+\sigma)[1]
 +     (\left(\begin{smallmatrix}
       1/k\\0
     \end{smallmatrix}\right) + \sigma)[\infty]
 \]
 with tailcone $\sigma$ being spanned by $e_1$ and $e_1+2k\cdot e_2$ if $n=2k$, see \cite[8.1]{zbMATH06868031}.
 These are of type $(1,2,n)$ and $(1,1,\sfrac{n}{2})$, respectively.  Lemma~\ref{ex:bound-or-not} implies that those singularieties do not belong to a bounded family.
 By lower-semicontinuity of the normalized volume~\cite[Theorem 2.11]{LLX20}, they all satisfy that 
\[
\widehat{\vol}(X_{n,1}) \geq 
\widehat{\vol}(X_0)=: v  >0.
\] 
Hence, we conclude that the family of $3$-dimensional
complexity one 
log terminal 
singularities 
$\{X_{n,1}\}_{n\in \nn}$
with normalized volume at least $v$
is bounded up to deformation,
but it is not bounded.

On the other hand, we had seen in Example~\ref{ex:running4} that $X_{n,1}$ is not K-semistable for $k\geq 5$.
 \end{example}

\begin{proof}[Proof of Theorem~\ref{thm:n-dim-K-ss-comp-1}]
  In Theorem~\ref{theorem:mld-bound} we showed that there is an upper bound for the minimal log discrepancies of torus invariant Koll\'ar components. Hence, the claim follows from Corollary~\ref{cor:semi-stable-bounded}.
\end{proof}

\begin{corollary}
  The set of $(2n-1)$-dimensional Sasaki-Einstein manifolds with an effective action of a torus of dimension at least $n-1$ and volume at least $v$ consists of a finite number of finite-dimensional families.
\end{corollary}

\section{Boundedness for K-semistable threefolds singularities}

The goal of this section is to prove Theorem \ref{introthm:3-dim}. We start from  some preparation.

 

\begin{lem}\label{lem:deg-alpha}
 Let $(V,B)$ be an $n$-dimensional klt log Fano pair where $B$ is a $\bQ$-divisor. Suppose $\alpha(V,B)<1$. Then there exists a non-trivial special degeneration $(V_0, B_0)$ of $(V, B)$ such that 
 \[
 \alpha(V_0, B_0)\geq \frac{1}{n}\min\{\alpha(V,B), 1-\alpha(V,B)\}.
 \]
\end{lem}

 \begin{proof}
 We separate into two cases. If $(V,B)$ is K-unstable, then by \cite{LXZ21} there exists a non-trivial special degeneration (called an optimal destabilization in \cite{BLZ19}) $(V_0,B_0)$ of $(V,B)$ such that $\delta(V_0,B_0)= \delta(V,B)$. Thus by \cite[Theorem C]{Blu18b} (see also \cite[Theorem A]{BJ17})  we have 
 \[
 \alpha(V_0, B_0) \geq \frac{1}{n+1}\delta(V_0, B_0) = \frac{1}{n+1}\delta(V, B) \geq \frac{1}{n+1}\cdot \frac{n+1}{n}\alpha(V,B) = \frac{1}{n}\alpha(V,B).
 \]
 
 From now on, we assume that $(V,B)$ is K-semistable. By \cite{Bir21}, there exists $D\sim_{\bQ} -K_V-B$ such that $\lct(V,B;D) = \alpha(V,B)=:a <1$. Let 
 \[
 b:=\sup\{t\geq 0\mid (V, B+tD)\textrm{ is a K-semistable log Fano pair}\}.
 \]

  We first estimate $b$. Let $F$ be a prime divisor over $V$ computing $\lct(V,B;D)$, i.e. $ a^{-1} A_{V,B}(F) = T_{V,B}(F) = \ord_F(D)$. By the proof of Lemma \ref{lem:delta-concave}, we know that 
 \begin{align*}
 \beta_{(V, B+tD)}(F)& = A_{V,B}(F) - t\ord_F(D) - (1-t) S_{V,B}(F)\\ & = (a-t) T_{V,B}(F) - (1-t) S_{V,B}(F)\\ & \leq  ((a-t)(n+1)-(1-t)) S_{V,B}(F).
 \end{align*}
 Thus if $1-t > (a-t)(n+1)$, then $(V, B+ tD)$ is K-unstable. This implies that 
 \begin{equation}\label{eq:a&b}
 b \leq \frac{(n+1)a - 1}{n}.
 \end{equation}
 In particular, we have $b<a<1$.
 
 Let $c\in (b,a)$ be a rational number. We express $c= (1-\epsilon) b + \epsilon a$. Then $(V, B+ c D)$ is K-unstable. Thus it admits a non-trivial special degeneration to $(V_0, B_0+cD_0)$ such that $\delta(V_0, B_0 + cD_0) = \delta(V, B+ cD)$. By Lemma \ref{lem:delta-concave} we know that
 \[
 \delta(V_0,B_0) \geq (1-c)\delta(V_0, B_0+cD_0)= (1-c)\delta(V, B+cD). 
 \]
 On the other hand, since $t\mapsto (1-t)\delta(V, B+tD)$ is concave by Lemma \ref{lem:delta-concave},  we have
 \[
 (1-c) \delta(V, B+cD) \geq (1-\epsilon) (1-b) \delta(V, B + b D)  + \epsilon (1-a) \delta(V, B+ aD) \geq (1-\epsilon) (1-b).
 \]
 Combining the two inequalities, we get
 \[
 \delta(V_0, B_0) \geq (1-\epsilon)(1-b). 
 \]
 Combining with the previous estimate \eqref{eq:a&b}, we get
 \[
 \alpha(V_0, B_0) \geq \frac{1}{n+1}\delta(V_0,B_0) \geq \frac{1}{n+1}(1-\epsilon) \left(1- \frac{(n+1)a - 1}{n}\right)=\frac{1-\epsilon}{n}(1-a).
 \]
 By boundedness \cite{Jia17, Che18, LLX20, XZ21} and discreteness of $\alpha$-invariants \cite{BLX19}, letting $\epsilon\to 0$ finishes the proof.
 \end{proof}

 \begin{lem}\label{lem:delta-concave}
Let $(V,B)$ be an $n$-dimensional klt log Fano pair where $B$ is a $\bQ$-divisor. Let $0\leq D\sim_{\bQ} -K_V-B$ be an effective $\bQ$-divisor. Then the function $t\mapsto (1-t) \delta(V, B+ tD)$ is non-increasing and concave for $t\in [0, \min \{1, \lct(V,B;D)\})$. 
 \end{lem}

 \begin{proof}
     Since $-K_V - B -tD \sim_{\bR}(1-t) (-K_V-B)$, we have $S_{(V, B+tD)}(E)=(1-t)S_{(V,B)}(E)$ for a prime divisor $E$ over $V$. Thus
     \[
     (1-t)\delta(V, B+tD) = (1-t) \inf_{E} \frac{A_{(V, B+tD)}(E)}{S_{(V, B+tD)}(E)} = \inf_E \frac{A_{(V, B)}(E)-t\ord_E(D)}{S_{(V, B)}(E)},
     \]
     where the infima run over all prime divisors $E$ over $V$. Since for each fixed $E$, the function $t\mapsto \frac{A_{(V, B)}(E)-t\ord_E(B)}{S_{(V, B)}(E)}$ is non-increasing and linear, we know that $t\mapsto (1-t) \delta(V, B+ tD)$ is non-increasing and concave.
 \end{proof}

 \begin{rem}
Under the same assumptions as Lemma \ref{lem:deg-alpha}, it was known that 
\[
\alpha(V_0,B_0)\leq \min\{\alpha(V,B), 1-\alpha(V,B)\}
\]
for every non-trivial special degeneration $(V_0, B_0)$ of $(V,B)$.
Indeed, the first inequality follows from the lower semicontinuity of $\alpha$-invariants \cite[Theorem B]{BL18}, while the second inequality follows from a logarithmic version of \cite[Proposition 3.3]{LZ19}.
 \end{rem}

\begin{proof}[Proof of Theorem~\ref{introthm:3-dim}]
If $(X,\Delta;\xi)$ has complexity at most $1$, this follows from Theorem \ref{thm:n-dim-K-ss-comp-1}. From now on we assume that $(X,\Delta;\xi)$ has complexity $2$, in particular, $\xi$ is rational. Then a rescaling of $\xi$ induces an effective $\bG_m$-action on $(X,\Delta)$. Let $E$ be the Koll\'ar component corresponding to $\xi$ . 
By \cite[Theorem 1.2]{Zhu21}, there exists a $\delta$-plt blow-up of $(X, \Delta;x)$ extracting $S$ where $\delta$ only depends on $\epsilon$ and $I$. Denote the different divisors of the Koll\'ar components $E$ and $S$ with respect to $(X,\Delta;x)$ by $\Delta_E$ and $\Delta_S$ respectively.
Since $(X,\Delta;\xi)$ is K-semistable, we know that $(E, \Delta_E)$ is also K-semistable, which implies that 
$\alpha(E,\Delta_E)\geq \frac{1}{3}$. 
If $S=E$, then the log boundedness of $(X,\Delta;\xi)$ follows from \cite{HLM20}. If $S\neq E$, then by the logarithmic version of \cite[Theorem 11]{CPS18} we have 
\[
\alpha(E,\Delta_E)+\alpha(S, \Delta_S)<1.
\]
Since $(S, \Delta_S)$ is $\delta$-klt, by \cite[Theorem 1.6]{Bir21} there exists $\alpha_0>0$ depending only on  $\delta$ (hence depending only on $\epsilon$ and $I$) such that $\alpha(S,\Delta_S)\geq \alpha_0$. For simplicity, we may assume that $\alpha_0\leq \frac{1}{3}$. By Lemma \ref{lem:deg-alpha}, there exists a non-trivial special degeneration $(E_0,\Delta_{E_0})$ of $(E, \Delta_E)$ such that 
\[
\alpha(E_0, \Delta_{E_0}) \geq \frac{1}{2}\min\{\alpha(E, \Delta_E), 1-\alpha(E, \Delta_E)\} \geq \frac{\alpha_0}{2}.
\]
This implies that $\delta(E_0, \Delta_{E_0}) \geq \frac{\alpha_0}{2}$. Since $(X,\Delta)$ is the orbifold cone over $(E,\Delta_E)$, the special degeneration $(E,\Delta_E)\rightsquigarrow (E_0,\Delta_{E_0})$ induces a non-trivial $\bG_m$-equivariant special degeneration $(X,\Delta;\xi)\rightsquigarrow (X_0,\Delta_0;\xi_0)$. Thus $(X_0, \Delta_0;\xi_0)$ carries an effective $\bT=\bG_m^2$-action where $\bG_m<\bT$ is a subtorus. Let $x_0$ be the cone point of $(X_0,\Delta_0;\xi_0)$. By  \cite[Proposition 7.5]{han2020acc} we have
\[
\hvol(X_0, \Delta_0; x_0) \geq \hvol_{(X_0,\Delta_0)}(\ord_{E_0})\cdot \left(\frac{\alpha_0}{2}\right)^3= \hvol_{(X,\Delta)}(\ord_{E})\cdot \left(\frac{\alpha_0}{2}\right)^3 > \frac{\epsilon \alpha_0^3}{8}.
\]
Since $(X_0, \Delta_0;\xi_0)$ has complexity at most $1$, by Theorem \ref{introthm:norm-vol-bound} we know that it admits a $\bT$-equivariant log bounded special degeneration $(X',\Delta';\xi')$. Composing these two special degenerations as \cite[Proof of Proposition 3.1]{LWX21}, we have that $(X,\Delta;\xi)$ admits a $\bG_m$-equivariant log bounded special degeneration. Thus we get log boundedness of $(X, \Delta;\xi)$ by Proposition \ref{prop:log-bounded-deformation+k-semistable} as it is K-semistable.
\end{proof}

\begin{proof}[Proof of Theorem \ref{introthm:vol-discrete}]
    It is known that $0$ is an accumulation point of the set of normalized volumes in any dimension, for instance, considering $n$-dimensional quotient singularities $\bA^n/\mu_m$ of type $\frac{1}{m}(1,1,\cdots,1)$, whose normalized volume equals $\frac{n^n}{m}$ by \cite[Example 7.1]{li2016stability} and hence converges to $0$ as $m\to \infty$. Thus it suffices to show that for any fixed $\epsilon>0$, there are only finitely many possible values of $\hvol(X,\Delta;x)$ that are larger than $\epsilon$. By Theorem \ref{thm:sdt}, the normalized volume minimizer $v$ induces a degeneration of $(X,\Delta;x)$ to a K-semistable singularity $(X_0, \Delta_0; \xi_v)$ where $x_0\in X_0$ denotes the cone point. Moreover, by Proposition \ref{prop:nvol-special-fibre} we know that 
    \[
    \epsilon< \hvol(X,\Delta;x) = \hvol_{(X,\Delta)}(v) = \hvol_{(X_0, \Delta_0)}(\xi_v) = \hvol(X_0, \Delta_0;x_0). 
    \]
    By Theorem \ref{introthm:3-dim}, all such K-semistable singularities $(X_0,\Delta_0;\xi_v)$ belong to a log bounded family, which implies that their normalized volumes belong to a finite set by \cite[Theorem 1.3]{Xu19}.
\end{proof}

\section{Boundedness for K-semistable hypersurface singularities}
\label{sec:hyper}
We complement our previous boundedness results  by one for K-semistable Fano cone hypersurface singularities.

Here, we consider a positive $\ZZ$-grading on the polynomial ring
\(
\CC[x_0,\ldots,x_n]
\)
with weights $0 < w_0 \leq w_1 \leq \ldots \leq w_n$. The hypersurface singularity $X_d$ is given by a weighted homogeneous equation of degree $d$. We may assume that each variable occurs in a non-linear monomial of the equation. Otherwise, the singularity would be either smooth or a product of a lower dimensional singularity and an affine space. The singularity $X_d$ is log terminal only if $\sum_i w_i-d >0$ and K-semistability is obstructed by the Lichnerowicz condition of Theorem~\ref{thm:lichnerowicz}, which demands that the inequality $\sum_iw_i-d \leq nw_0$ holds. In the case that $X_d$ is K-semistable with respect to the polarization given by the grading, the normalized volume of $X_d$ is calculated by the valutation corresponding to the $\ZZ$-grading. Here, we have
\[
\nvol(X_d;x)=\frac{d(\sum_iw_i -d)^n}{w_0\cdots w_n}.
\] 

In order, to prove boundedness for K-semistable Fano cone hypersurface singularities with volume bounded by $v > 0$ we consider the tuples $(w_0,\ldots,w_n,d) \in \NN^{n+2}$, up to taking multiples, which fulfil the following set of conditions.
\begin{align}
\sum_i w_i -d \geq 0& & \text{(log terminality)}\label{eq:log-terminal}\\
\sum_i w_i -d \leq nw_0& & \text{(Lichnerowicz obstruction)}\label{eq:lichnerowicz}\\
w_0 + w_n \leq d  && \text{(every variable appears in a non-linear monomial)}\label{eq:dim=n} \\
\frac{d(\sum_iw_i -d)^n}{w_0\cdots w_n} \geq v& & \text{(bound on the volume)}\label{eq:volume-bound}
\end{align}
As discussed in Example~\ref{exp:hypersurface2} of the introduction, we may extend this setting to gradings by some semigroup in $\RR_{\geq 0}$, when we also allow positive real values for the weights and the degree. Equivalently, this means that the equation is actually homogeneous with respect to a (multi-)grading by some $M=\ZZ^r$ and our $\RR$-weights are obtained by evaluating a (potentially irrational) element $\xi \in M_\RR^*=(\RR^r)^*$ in the $M$-weights of the variables. By approximating $\xi$ by rational elements, we see that all conditions still have to hold in this generalized setting.

\begin{proof}[Proof of Theorem~\ref{introthm:hypersurfaces}]
  Since we consider the tuples $(w_0,\ldots,w_n,d)$ only up to scaling, we may normalize them such that $w_0=1$ holds.  Hence, we from now on assume that $w_0=1$ and allow real values $\geq 1$ for the other weights. We will call a such a tuple \emph{weight data}. 
  
  Consider the set of weight data corresponding to $n$-dimensional K-semistable polarized Fano cone singularities. In particular, the weight data has to fulfil all our conditions above.  Our aim is to show that $d$ is bounded by some constant. Then we would be done, as there are only finitely many monomials having degree less or equal to $d$ (since the weights always fulfil $w_i \geq 1$). If $d$ is not bounded, then there would exist an infinite sequence of weight data, such that $d\to \infty$.  Assume there is such a sequence. In order to keep the notation simple, we refrain from adding a new index to the $w_i$ and $d$ just to indicate the membership to that sequence.  

  Recall that the normalized volume \[\nvol(X_d)=\frac{d(\sum_iw_i -d)^n}{w_0\cdots w_n} =\frac{\left(\sum_iw_i-d \right)^n}{\frac{w_0\cdots w_n}{d}}\geq v\] associated to our weight data is bounded from below by (\ref{eq:volume-bound}). 
  Since $\sum_iw_i-d \leq n$ by (\ref{eq:lichnerowicz}),
  then we have that $\frac{w_0\cdots w_n}{d}\leq n^n/v$. Hence, after passing to a subsequence we may assume $\frac{w_0\cdots w_n}{d} \to c$ for some $c \geq 0$. Similarly, after passing to a subsequence,  we may assume that each $w_i$ is either bounded or diverges to $+\infty$. 
  By inequality (\ref{eq:log-terminal}), at least one of the $w_i$ must diverge to $+\infty$.
  
\begin{mycases}

\litem{$d \to \infty$, $w_n \to \infty$, and $w_{n-1} \to \infty$}{
Because of the convergence $\frac{w_0\cdots w_n}{d} \to c$, we eventually have  $d > \frac{w_0\cdots w_n}{1+c}$ in our sequence of weight data.
On the other hand, 
the fact that $w_n\to \infty$
and $w_{n-1} \to \infty$ implies that eventually we have that
\[
\sum_i w_i < \frac{1}{1+c}\prod_i w_i. 
\]
Hence, at some point, we obtain 
\begin{equation*}
  \sum_i w_i - d  < \sum_i w_i - \frac{1}{1+c} \prod_i w_i < 0,
\end{equation*}
but this contradicts the log terminality condition (\ref{eq:log-terminal}).}\\

Note that $1\leq w_2\leq\dots\leq w_{n-1}\leq w_n$ for any sequence data.
By the first case, we can assume that $w_{n-1}\to 1$.
Hence, we may assume that 
$w_i \to 1$ for every $i\in \{0,\dots,n-1\}$.
  
\litem{$d \to \infty$, $w_n \to \infty$  and $w_0,\ldots, w_{n-1}$ are bounded}{
We have $\sfrac{w_n}{d} < 1$ by (\ref{eq:dim=n}).  By assumption $\sum_{i=0}^{n-1} w_i$ is bounded from above by some constant $C$ and we obtain
  \[\sum_{i=0}^n w_i-d = \sum_{i=0}^{n-1} w_i +\frac{w_n}{d}\cdot d-d < C - \left(1-\frac{w_n}{d}\right)d.\] Since $d \to +\infty$ and  $\sum_{i=0}^n w_i -d > 0$ by (\ref{eq:log-terminal}) it follows that $\sfrac{w_n}{d}\to 1$.

  Now, for $w_n > \sfrac{d}{2}$ an equation of degree $d$ will look like
  \begin{equation*}
    \label{eq:equation-wdata}
    x_ng_{d-w_n}(x_0,\ldots, x_{n-1}) + g_{d}(x_0,\ldots, x_{n-1}),
  \end{equation*}
  where $g_{d-w_n}$ and $g_{d}$ are quasi-homogeneous polynomials of the corresponding degrees. We will now construct a special degeneration for such a singularity $X_d$. For this we embed the singularity as intersection of two hypersurfaces into $\mathbb A^{n+2}$:
  \[
    x_nx_{n+1} + g_{d}(x_0,\ldots, x_{n-1}) = x_{n+1} - g_{d-w_n}(x_0,\ldots, x_{n-1})=0.
  \]
  These equations becomes quasi-homoegenous of degree $d$ and $d-w_n$, respectively, if we choose the weight $w_{n+1}=d-w_n< w_n$ for $x_{n+1}$. Now, consider the special degeneration given by
  \[
    x_nx_{n+1} + g_{d}(x_0,\ldots, x_{n-1}) = t^a x_{n+1} - g_{d-w_n}(x_0,\ldots, x_{n-1})=0.
  \]
  Here, we act on the coordinates by the weights $(w_0,\ldots,w_{n-1},w_n-a,w_{n+1}+a)$. The central fibre of the degenration is cut out by
  \[
    x_nx_{n+1} + g_{d}(x_0,\ldots, x_{n-1}) = g_{d-w_n}(x_0,\ldots, x_{n-1})=0.
  \]
  Let $v$ be the valuation on $X_d$ corresponding to the special degeneration. Then we obtain
  \[
    \nvol_X(v) = \frac{d(d-w)\left(\sum_{i=0}^{n} w_i -d \right)^n}{w_0\cdots (w_n-a) \cdot (w_{n+1}+a)}  \leq \frac{d(d-w)n^n}{w_0\cdots (w_n-a) \cdot (w_{n+1}+a)}.
  \]
  Here, the inequality on the right follows from~(\ref{eq:lichnerowicz}). Now, by choosing $a=\sfrac{w_n}{2}$ and observing that $\nvol(X_d;x) \leq \nvol_{X_d}(v)$ we obtain
  \begin{equation*}
    \nvol(X_d;x) \leq  \frac{d(d-w)n^n}{w_0\cdots w_{n-1} (\sfrac{w_n}{2}) \cdot (d-\sfrac{w_n}{2})}
    =\frac{\sfrac{d}{w_n}(\sfrac{d}{w_n}-1)n^n}{w_0\cdots w_{n-1} \cdot \sfrac{1}{2} \cdot (\sfrac{d}{w_n}-\sfrac{1}{2})}.\label{eq:limit-volumes}    
  \end{equation*}
  From the fact that $w_0, \ldots, w_{n-1} \geq 1$ and $\sfrac{d}{w_n} \to 1$ we see that $\nvol(X_d;x) \to 0$, but this is in contradiction to (\ref{eq:volume-bound}). Alternatively, we could also argue that
  $\nvol(v) < \nvol(\xi)=\frac{d(\sum_i w_i -d)^n}{x_0 \cdots x_n}$ and therefore $(X_d, \xi)$ is not K-semistable, which is in contradiction to our other assumption.
}

\end{mycases}

\end{proof}

\bibliographystyle{habbrv}
\bibliography{volbound}

\end{document}